\documentclass[12pt, a4paper]{amsart}

\usepackage{amsmath, amsthm, amssymb}
\usepackage{graphicx}
\usepackage{enumerate}
\usepackage[top=2.5cm, bottom=2.5cm, left=2.5cm, right=2.5cm]{geometry}
\usepackage{url}

\newtheorem{lemma}{\bf Lemma}

\newtheorem{theorem}{\bf Theorem}

\renewenvironment{proof}{\noindent {\bf Proof: }}{\rm\\}
\theoremstyle{definition}
{\rm}

\usepackage{color}

\usepackage[T1]{fontenc}

\usepackage{algorithmic}
\usepackage{algorithm}

\makeatletter
\renewcommand{\p@algorithm}{\arabic{algorithm}\expandafter\@gobble}
\makeatother

\newcounter{step}[algorithm]
\newcommand\STEP[2][\(\triangleright\)]{%
	\refstepcounter{step}
	\vskip 0.25\baselineskip
	\item[]\hskip -\algorithmicindent #1 \textbf{Step \arabic{step}}%
	\ifthenelse{\equal{\unexpanded{#2}}{}}{}{ (\texttt{#2})}%
	\textbf{.}%
}
\newenvironment{algo}{\algo}{}
\def\algo#1\end{%
	\noindent\fbox{%
	\begin{minipage}[b]{\dimexpr\columnwidth-\algorithmicindent\relax}
	\begin{algorithmic}
	#1
	\end{algorithmic}
	\end{minipage}
	}%
\end}

\usepackage{mathrsfs}

\begin{document}
\title{Optimizing the  Kreiss constant}

\author{Pierre Apkarian$^1$}
\author{Dominikus Noll$^{2}$}
\thanks{$^1$ONERA, Information Processing and Systems, Toulouse, France}
\thanks{$^2$Institut de Math\'ematiques, Universit\'e de Toulouse, France}


\maketitle

\begin{abstract} 
The Kreiss constant $K(A)$ of a stable matrix $A$ conveys information about the transient 
behavior of system trajectories in response to  initial conditions. 
We present an efficient way to compute the Kreiss constant $K(A)$, and we show how
feedback can be employed to make the Kreiss constant $K(A_{cl})$ in closed loop significantly smaller.
This is expected to reduce transients in the closed loop trajectories. The proposed approached is compared to potential competing techniques.  

\vspace{.2cm}


\end{abstract}

\section{Introduction}

Given a stable autonomous system   
\begin{equation}\label{eqPlant1}
\dot x = A x, \quad x(0):= x_0, \quad A \in \mathbb R^{n\times n}\,,
\end{equation}
the time-dependent worst-case transient growth of the trajectories in response to initial conditions $x_0$
is 
$$
\max_{\|x_0\| = 1} \|e^{At} x_0\|= \|e^{At}\| \,,
$$
where $\|.\|$ denotes both the vector 2-norm and the induced
spectral matrix norm or maximum singular value norm. The maximum transient growth, or {\it transient growth} for short, is then the quantity
\begin{equation}
    \label{growth}
M_0(A) =\sup_{t \geq 0} \|e^{At}\|,
\end{equation}
which gives information about the maximum amplification
of system responses to all possible initial conditions at all times.

The Kreiss constant $K(A)$ of the matrix $A\in \mathbb R^{n\times n}$ may be introduced by means of its resolvent
as
\begin{equation}\label{eqKr1}
K(A):= \max_{{\rm Re} (s) > 0}\; {\rm Re} (s) \, \|(sI-A)^{-1} \| \,,
\end{equation}
and its importance is due to the Kreiss Matrix Theorem \cite[p. 151, p. 183]{trefethen2005spectra}, which relates it to the transient growth $M_0(A)$ by
providing lower and upper bounds:
\begin{equation}\label{eqBounds}
K(A) \leq M_0(A)= \sup_{t \geq 0} \|e^{At}\| \leq e\,n\,  K(A)\,,
\end{equation}
where $e=2.7183...$ is the Euler number. Alternatively,
the Kreiss constant has also the representation
\begin{equation}
K(A) = \sup_{\epsilon > 0} \frac{\alpha_\epsilon(A)}{\epsilon},
\end{equation}
where $\alpha_\epsilon(A)$ is the $\epsilon$-pseudo spectral abscissa \cite{trefethen2005spectra}.

The Kreiss constant was originally introduced in the discrete setting as an analytic tool to assess stability
of numerical schemes \cite{kreiss1962stabilitatsdefinition}. Since then it has 
manifested itself as a quantitative measure 
of non-normal behavior of matrices \cite{trefethen2005spectra,asllani2018structure}, owing to the fact
that $K(A)\geq 1$, with equality e.g. if $A$ is normal.  More precisely, the global minimum $K(A) = 1$ is attained if and only if 
$M_0(A)=1$ attains its global minimum, which is at those matrices $A$ where
$e^{At}$ is a contraction in the spectral norm. 
Outside the realm of dynamical systems, this quantitative aspect of $K(A)$  has for instance been of interest in the analysis of networks \cite{asllani2018structure}. 

Even though our principal concern here is with matrices, it is worthwhile having a look at the case
of $C_0$-operator semi-groups. Here 
the left hand estimate $K(A) \leq M_0(A)$ from (\ref{eqBounds}) is still valid, as is the observation that
$K(A)=1$ implies $M_0(A)=1$, with the global minimum attained at least in Hilbert space 
for contraction semi-groups in the spectral norm. Both facts are easy consequences of the Hille-Yoshida theorem \cite{engel_nagel2000}.
The conclusion is that even for semi-groups the transient dynamics are suitably assessed through the Kreiss constant.


While the Kreiss constant $K(A)$ has received ample attention in numerous books, treatises and articles 
as a theoretical quantity to
analyze transient system behavior, \cite{trefethen2005spectra}, 
its computation has only very recently been addressed. In \cite{mitchell2019computing} the author  
uses a variety of  local optimization techniques in tandem with global searches to compute $K(A)$ with certificates. 
In \cite{trefethen2005spectra},  $K(A)$ is simply estimated graphically by plotting the ratio
$\alpha_\epsilon(A)/\epsilon$ against $\epsilon$ and searching for the maximum, and this seems to have been pioneered in \cite{mengi2006measures}.

In the present paper, we show that the Kreiss constant $K(A)$ can be computed exactly with limited complexity using
techniques from robust control. Our new characterization opens the way to more challenging situations,  
where the Kreiss constant is not just computed, but more ambitiously,
minimized in closed loop with the goal  to constrain the  transient growth of a plant (\ref{eqPlant1}) 
by the use of feedback. For short, one may wish to use feedback to bring the closed-loop
$A_{cl}$ closer to contractive transient behavior  than the original matrix $A$.

This is expected to have consequences in feedback control of non-linear systems,
where it is known that non-normality of the system Jacobian at steady state may lead to large  transient amplifications even for well-damped spectra, 
which trigger non-linear effects responsible for  instability, or lead to undesirable limit-cycle  dynamics.
This phenomenon is well known in the fluid dynamic community 
\cite{Leclercq2019,schmid2014analysis,sipp2010dynamics,trefethen1993hydrodynamic,reddy1993energy}.

The structure of the paper is as follows. In section \ref{sect:exact}, we obtain a formula for $K(A)$
which can be used to compute it with reasonable effort, by relating it to the structured singular value or $\mu$  known in
robust system analysis. In section \ref{sect:feedback} we widen the scope and address the problem of
minimizing $K(A_{cl})$ in closed loop. Since this is an NP-hard problem, a fast heuristic is presented,
which is based on non-differential optimization techniques. Section \ref{sect:optim} gives a short overview of these
techniques, and shows how the result of the local optimization can be certified using the techniques of section \ref{sect:exact}.
Numerical experiments and additional concurrent techniques are presented in section \ref{sect:numerics}.

\section*{\sc Notation \label{sect-nota}}
For complex matrices $X^H$ stands for the conjugate transpose. 
%
The terminology  follows \cite{ZDG:96}.
Given   partitioned matrices  $$M:= 
                                    \begin{bmatrix}
                                      M_{11} & M_{12} \\
                                      M_{21} & M_{22} \\
                                    \end{bmatrix}
                                     \;\mbox{ and }
N:= 
                                    \begin{bmatrix}
                                      N_{11} & N_{12} \\
                                      N_{21} & N_{22} \\
                                    \end{bmatrix}                                
$$
of appropriate dimensions  and assuming existence of inverses, the Redheffer star product \cite{doyle91_2,R:1960} 
of $M$ and $N$ is  $M\star N := $
$$
              \begin{bmatrix}
                 M \star N_{11}& M_{12}(I-N_{11}M_{22})^{-1} N_{12} \\
                N_{21}(I-M_{22}N_{11})^{-1} M_{21} &  N \star M_{22} \\
              \end{bmatrix}.
$$

When $M$ or $N$ do not have an explicit $2\times 2$ structure, 
we assume consistently that the star product reduces
to a linear fractional transform (LFT). The lower LFT of $M$
and $N$ is denoted $M\star N$ and defined as
$$M\star N := M_{11} + M_{12} N (I-M_{22}N)^{-1} M_{21},$$
and the upper LFT of $M$ and $N$ is denoted $N\star M$ and obtained as
$$N\star M:= M_{22} + M_{21} N (I-M_{11}N)^{-1} M_{12}\,. $$
With these definitions, the $\star$ operator is associative. 

\section{Exact computation of the Kreiss constant \label{sect:exact}}
It is readily seen from (\ref{eqBounds}) that the  Kreiss constant is finite if 
system (\ref{eqPlant1}) is stable, that is, has strictly negative spectral abscissa $\alpha(A) < 0$. 
When unstable matrices are concerned, it is convenient to consider translated bounds, cf. \cite{trefethen2005spectra}, which 
correspond to shifting the matrix $A$ to stability, e.g.  by its spectral abscissa.
For the rest of the paper the symbol $K(A)$ will therefore only be used when $A$ is stable.
\begin{theorem}
\label{theo1} 
The Kreiss constant $K(A)$ can be computed through the robust $H_\infty$-performance analysis program
\begin{align}
    \label{eqRob}
K(A) &= \max_{\delta \in [-1,\,1]} \left\| \left(s I - \left( \textstyle\frac{1-\delta}{1+\delta}  A -I\right)\right)^{-1}\right\|_\infty 
= \max_{\delta\in [-1,1]}\max_{\omega\in [0,\infty]} \overline{\sigma}\left(\left( j\omega I - \left( \textstyle\frac{1-\delta}{1+\delta}  A -I\right)\right)^{-1} \right).
\end{align}
\end{theorem}
\begin{proof}
Note that for $\delta = -1$ the expression between the norm signs is understood to denote the zero matrix, which contributes only the value
$0$ to the maximization. 

Starting with $s:= x + j y$ in (\ref{eqKr1}) gives
\begin{align*}
    K(A) & =  \sup_{x >0,y} x \|( (x +j y)I-A)^{-1} \| 
     =  \sup_{x >0,y} \left\| \left(\left(1+ j \textstyle\frac{y}{x}\right) I -\textstyle\frac{1}{x}A\right)^{-1}\right\|.
\end{align*}
The  change of variables $(y/x,1/x) = (\omega, \frac{1-\delta}{1+\delta}) $  is a bijective mapping 
from $\mathbb R \times \mathbb R_+$ into $\mathbb R \times (-1,\,1]$ and leads to the characterization
in (\ref{eqRob}). 
\hfill $\square$
\end{proof}

For future use we express program (\ref{eqRob}) using the Redheffer star product or equivalently the upper LFT
(see e.g. \cite{YoD:90,doyle91-1}): 
\begin{equation}
    \label{lft}
K(A) = \max_{\omega\in [0,\infty]} \max_{\delta\in [-1,1]} \overline{\sigma}\left(j\omega I-\left( \left( \delta I_n  \star Q\right) A -I\right)^{-1}  \right),
\end{equation}
where
$$
Q = \begin{bmatrix} -I_n &\sqrt{2} I_n \\-\sqrt{2}I_n& I_n\end{bmatrix}\,,
$$
and where $\delta I_n \star Q$ is understood as of repeating the rational term $n$ times.

As one notices the computation of (\ref{eqRob}) involves two global maximization
steps, one over the frequency axis $\omega\in \mathbb R$, and one over the uncertainty $\delta\in [-1,1]$, which can be performed
in either order.
This leads to two strategies, which will
both be exploited in this text.

The interpretation of (\ref{eqRob}) is that of a
transfer function $T_{wz}(s,\delta)$ with uncertainty $\delta\in [-1,1]$, where the worst-case
$H_\infty$-norm $\max_{\delta \in [-1,1]} \|T_{wz}(\cdot,\delta)\|_\infty$ has to be computed.  
In order to highlight this,
we represent the situation in state-space using the plant:
\begin{equation}\label{eqP}
P(s): \quad \left\{
\begin{aligned}
    \dot{x} &= Ax - x + \sqrt{2} w_\delta + w \\
    z_\delta &= -\sqrt{2} Ax - w_\delta \\
    z &= x 
\end{aligned} \right.
\end{equation}
which represents the transfer function form $(w_\delta,w)$ to $(z_\delta,z)$, and 
which is in upper feedback with the block $w_\delta = \delta z_\delta$,
leading to $T_{wz}(\cdot,\delta) = \delta I_n\star P$ and giving the Redheffer representation
\begin{equation}
\label{eqRob3}
K(A) = \max_{\delta \in [-1,1]} \| \delta I_n \star P\|_\infty
\end{equation}
of the Kreiss constant as a worst-case $H_\infty$-norm. See Fig. \ref{figDiag}.

\begin{figure}[htp]
\centering
$\qquad\qquad$\includegraphics[scale=0.6]{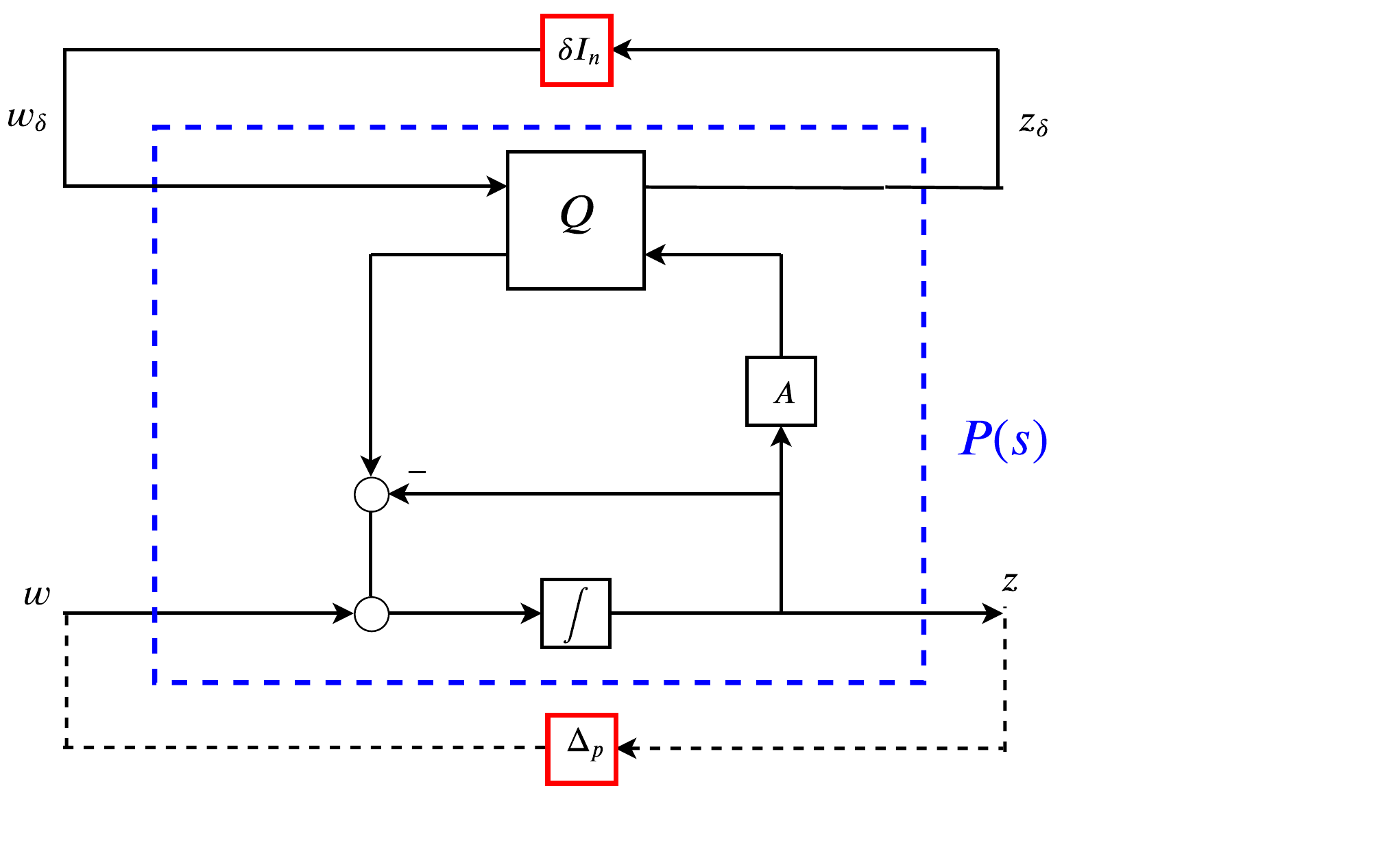}
\caption{Diagram representation of Kreiss constant} 
\label{figDiag}
\end{figure}

Formula (\ref{lft}) leads to a different approach.  Namely,
as is common in robustness analysis, the performance channel $w\to z$ can be replaced with  a fictitious full block $\Delta_p\in \mathbb C^{n\times n}$, leading to a specially structured robust stability problem. See  Fig. \ref{figDiag}. 
The problem has now two blocks and can be addressed using  the 
structured singular value (SSV) or $\mu$-singular value \cite{Young:96,doyle91-1,YoD:90}. Recall that
for a complex matrix $M$ and a structure $\mathbf \Delta$ of uncertain matrices $\Delta$,  $\mu_{\mathbf \Delta}(M)$ is defined as 
$$
\mu_{\mathbf \Delta}(M):= \frac{1}{ \inf \left\{\|\Delta\|:\, \Delta \in \mathbf \Delta,\, \det (I-M\Delta) = 0\right\}},
$$
where as usual inf$\,\emptyset =+\infty$, so that $\mu_{\mathbf \Delta}(M)=0$ if no $\Delta \in \mathbf \Delta$ makes $I-M\Delta$ singular.

In our case the structured singular value is computed with respect to the structure
${\bf \Delta} = \{{\rm diag}(\delta I_n,\Delta_p): \delta \in \mathbb R, \Delta_p \in \mathbb C^{n\times n}\}$.
We have by \cite[Thm. 11.9]{ZDG:96}:
\begin{lemma}
\label{lemma1}
Let $\omega$ be fixed. The following statements {\rm 1.} and {\rm 2.} are equivalent:
\begin{enumerate}
    \item[\rm 1.] \begin{itemize} \item[(i)] $\displaystyle \delta I_n \star P(j\omega)$ is well-posed over $[-1,1]$ and \\[-0.1cm]
\item[(ii)] $\displaystyle\max_{\delta \in [-1,1]}  \overline{\sigma}\left(\delta I_n \star P(j\omega) \right) <  \gamma$ \\[0.5cm]\end{itemize}
\item[\rm 2.]
$\mu_{\bf \Delta}\left(P(j\omega)\begin{bmatrix} I_n&0\\0&I_n/\gamma \end{bmatrix}    \right) < 1.$ 
\end{enumerate}
\hfill $\square$
\end{lemma}

This implies
the following:

\begin{theorem}\label{theo2} For any fixed $\omega$, the optimal value of the inner program of {\rm (\ref{lft})} is obtained with arbitrary
precision $\epsilon > 0$ as  the value
of the one-dimensional optimization program
\begin{eqnarray}
\label{eqpg1}
\begin{array}{ll}
\mbox{\rm minimize} & \gamma \\
\mbox{\rm subject to} & \mu_{\bf \Delta}\left(P(j\omega) \begin{bmatrix} I_n&0\\0&I_n/\gamma  \end{bmatrix}  \right) \leq 1-\epsilon
\end{array}
\end{eqnarray}
where the structured singular value $\mu_{\bf \Delta}$ is computed with respect to the block structure ${\rm diag} (\delta I_n, \Delta_p)$ with $\delta$ real, and 
$\Delta_p \in  \mathbb C^{n\times n}$.  
\hfill $\square$
\end{theorem}
Since the constraint 2. in Lemma \ref{lemma1} has to be satisfied strictly in order to assure robust stability,
$\mu_{\bf \Delta}<1$ had to be replaced by $\mu_{\bf \Delta}\leq 1-\epsilon$ in program (\ref{eqpg1}) for an arbitrarily small $\epsilon > 0$.

It is well-known that the evaluation of the structured singular value $\mu_{\bf \Delta}$ is in general NP-hard \cite{toker1995spl,braatz1994computational},
so that the constraint in (\ref{eqpg1}) may appear intractable.  This is why $\mu_{\bf \Delta}$ is in usually replaced by 
its $\mu$-upper bound $\overline{\mu}_{\bf \Delta}(M)$,
where in general only $\mu_{\bf \Delta} <  \overline{\mu}_{\bf \Delta}$.
However, there are five cases, where the upper bound is exact, and presently we have one of these five, because
${\bf \Delta}$ consists of only one repeated real block and a single full complex block; see \cite[p. 282]{ZDG:96}.
See also the elegant derivation in \cite{MSF:1997}. This means the constraint in (\ref{eqpg1}) is computable exactly by a linear matrix inequality or a convex SDP. We have
\begin{theorem}\label{theo3}
For fixed $\omega$, the optimal value of the inner optimization program in {\rm (\ref{lft})} may be obtained by the following
convex semi-definite program (SDP):
\begin{equation}
\label{eqpg2}
\begin{array}{ll}
\mbox{\rm minimize}& \gamma \\
\mbox{\rm subject to} &   X,Y \in \mathbb C^{n\times n}, X= X^H, Y^H = -Y ,\gamma \in \mathbb R\\
&
\!\!\!\!\!\!\!\!\!\!\!\!\!\!\!\!\begin{bmatrix}\bullet \end{bmatrix}^H
\begin{bmatrix}X & 0 & Y & 0 \\
0 & I_n & 0 & 0 \\
Y^H & 0 & -X & 0 \\
0 &0 &0 &-\gamma^2 I_n \end{bmatrix} \begin{bmatrix} P(j\omega)\\ I_{2n}\end{bmatrix} \preceq -\epsilon I
\end{array}
\end{equation}
\end{theorem}

\begin{proof} The cast (\ref{eqpg1}) is a direct consequence of the Main Loop Theorem \cite{ZDG:96}. Program (\ref{eqpg2}) 
computes the $\mu_{\bf \Delta}$ upper-bound \cite{doyle91-1,ZDG:96}, but since
for the specific block structure involving one repeated parameter $\delta$ and a single 
complex full block the upper bound is exact, this now coincides with the true value of $\mu_{\bf \Delta}$  \cite{MSF:1997}.
\hfill $\square$
\end{proof}

Since program (\ref{eqRob3}) can be solved exactly at any given frequency, one is left with a search over the frequency axis.  A
straightforward idea would appear to be frequency gridding, but a more advisable approach is based on dividing the frequency axis 
into intervals, on each of which the Hamiltonian test can be applied  \cite{ferreres2003robustness,Lawrence:2000,sideris1990elimination}. 

In summary, the above results show that the Kreiss constant can be computed to any prescribed accuracy using fairly standard robust analysis techniques. 
\\

\noindent {\bf Example.}
As simple test set, we consider Grcar (named after Joseph Grcar)  matrices of various dimensions and estimate the Kreiss constant
using either the method of Theorem \ref{theo1} or the one in Theorem in \ref{theo3}. 
The Grcar matrices considered here are band-Toeplitz  matrices with the first subdiagonal and main diagonal set to $-1$   and $3$ superdiagonals set to $1$ and zero entries elsewhere. Such matrices are known to possess very sensitive eigenvalues and therefore deviate from normality. 

\begin{table}[htbp]
\begin{center}
\caption{Kreiss constant estimates and running times (sec.) based on  Theorems \ref{theo1} and \ref{theo3}. I: impractical \label{tabGrcar}}
\begin{tabular}{||c||c|c||c|c||}
\hline\hline
size &  \multicolumn{2}{c||}{method of Theorem \ref{theo1} }  & \multicolumn{2}{c||}{method of Theorem \ref{theo3}} \\ \hline
 & estimate & cpu & estimate & cpu  \\ \hline
10  & 1.1855e+00   &   2  & 1.1881e+00 &  2 \\ \hline
20  & 2.7199e+00   &   4  & 2.7255e+00 &  68 \\ \hline
30  & 8.7803e+00   &   7  & 8.7989e+00 &  720 \\ \hline
40  & 3.3155e+01   &   12  & 3.3223e+01 &  6800 \\ \hline
50  & 1.3548e+02   &   22  & 1.3577e+02  & 30968 \\ \hline
100 & 2.4837+e05   &   127 & I & I \\ \hline\hline
\end{tabular}
\end{center}
\end{table}
Estimates of the Kreiss constant for problems of increasing size are given in table \ref{tabGrcar}.
We observe that while the worst-case $H_\infty$-norm approach in Theorem \ref{theo1} is operational for medium size problems, the $\mu$ certificate
based on Theorem \ref{theo3} becomes quickly impractical which is an incentive to develop dedicated  methods. For the case $n = 50$, the $H_\infty$ norm vs. $\delta$  and the transient growth $\|e^{At}\|$ are presented in Fig. \ref{figHTG}. Note the shape and peak value $135.5$ of the left curve in Fig. \ref{figHTG} are consistent with the results in \cite{mengi2006measures} based on  $\displaystyle f(\epsilon):= \alpha_\epsilon\slash\epsilon$ with estimated peak value  of $133.6$. 

\begin{figure}[htp]\label{HTG}
\centering
\includegraphics[width = 0.45\columnwidth]{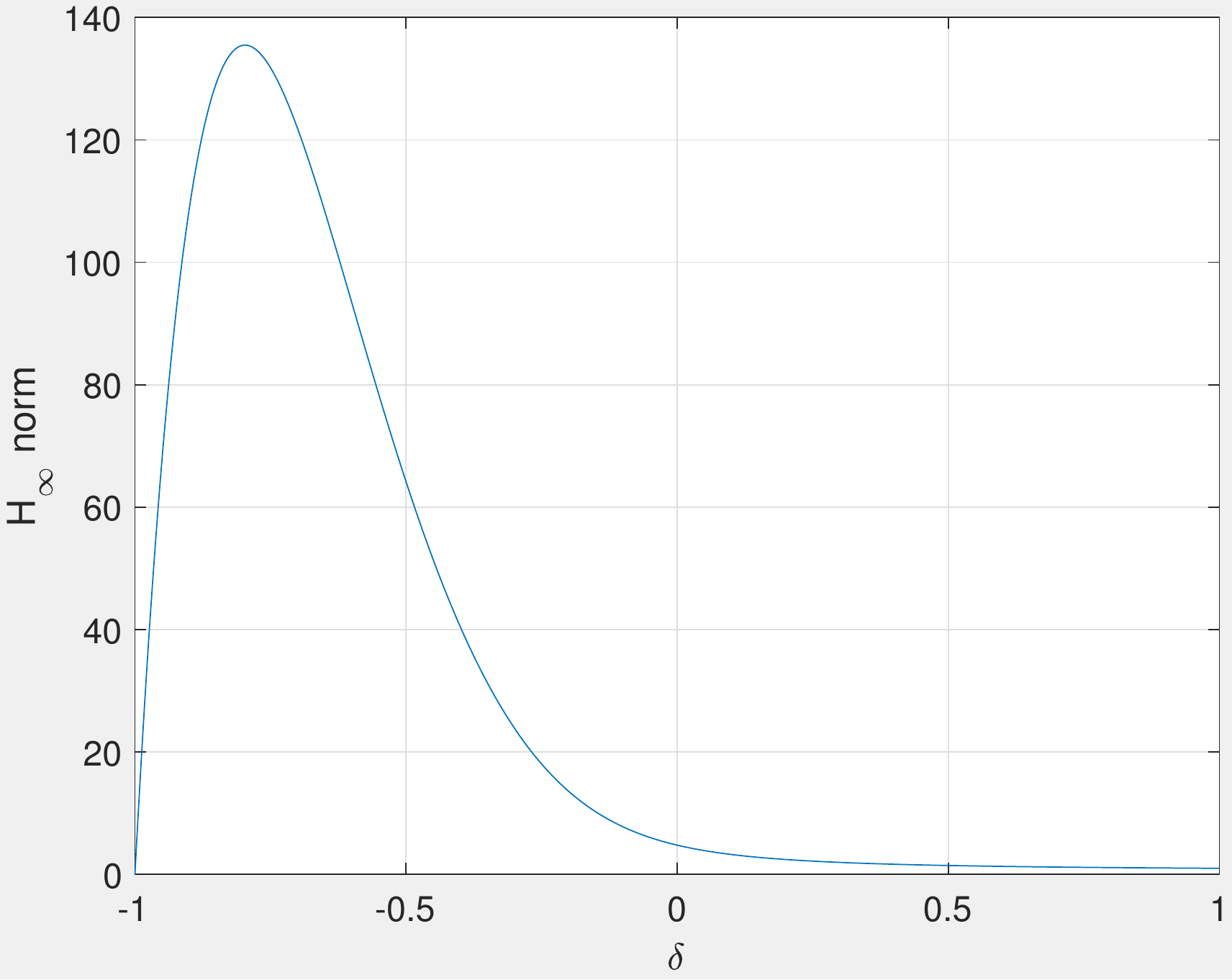}$\;$
\includegraphics[width = 0.45\columnwidth]{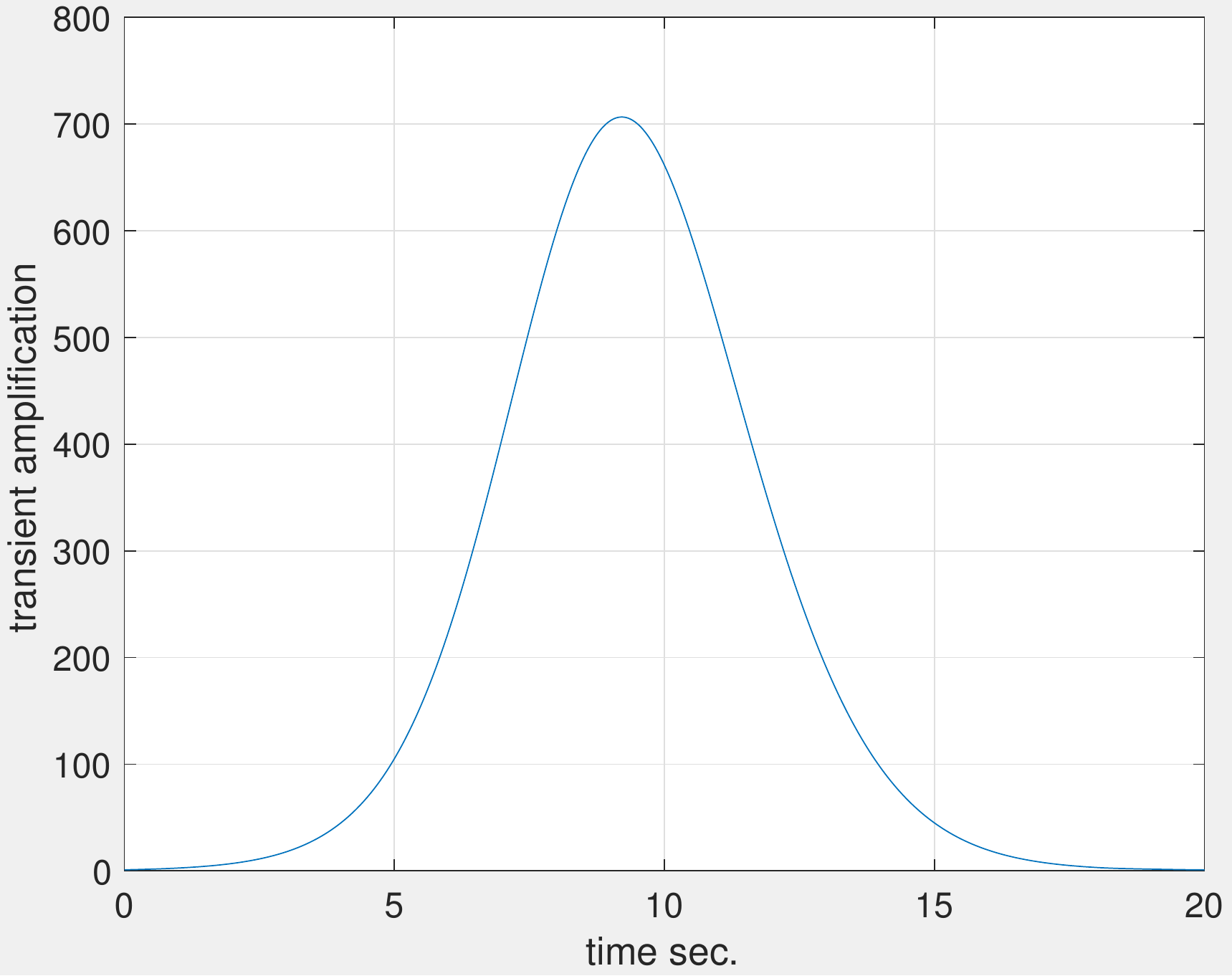}
\caption{Left: $H_\infty$ norm vs. $\delta$,$\qquad$ Right: Transient growth} 
\label{figHTG}
\end{figure}



\section{Feedback control of transient growth \label{sect:feedback}}
In this section, we further explore the Kreiss constant and its link to transient growth 
by employing feedback to reduce it in closed loop.  
Consider a plant $G(s)$ with control inputs $u\in \mathbb R^m$ and outputs measurements $y\in \mathbb R^p$:
\begin{align}\label{eq-ABCD}
\begin{split}
\dot x & =  A x + B u,\qquad x \in \mathbb R^n\\ 
y & =  C x + D u\,,
\end{split}
\end{align}
in loop with either a static feedback controller $K \in \mathbb R^{m\times p}$ giving $u = K y$, or
a dynamic output-feedback controller $K(s)$ giving 
\begin{align}
\begin{split}
\dot x_K & =  A_K x_K+ B_K y,\qquad x_K \in \mathbb R^{n_K} \\
u & =  C_K x_K + D_K y\,.
\end{split}
\end{align}
We make the  assumption $D=0$, which incurs no loss of generality, while considerably simplifying the presentation. 
The closed-loop autonomous system is described as
$$
\dot x_{cl} = A_{cl} x_{cl},\quad x_{cl}(0) = x_{cl}^0\,,
$$
with state matrix $A_{cl}$ obtained in both cases  as
$$
A_{cl} = A + B K C \mbox{ or }\quad  A_{cl}=\begin{bmatrix}
A + B D_KC & B C_K \\ B_K C & A_K
\end{bmatrix}\,.
$$
The transient growth of the closed loop may now be
assessed either by $M_0(A_{cl})$, or by concentrating on the plant state trajectories $x(t)$ generated by initial conditions $x_0$. 
The latter are described by
\begin{align*}
\mathcal M_0(A_{cl}) &= 
\sup_{t\geq 0} \max_{\|x_0\|=1} \|
J^T
e^{A_{cl}t}J 
x_0 \| = \sup_{t\geq 0} \|J^T
e^{A_{cl}t}J \|
\end{align*}
where $J:=I_n$ for a static output feedback controller and $J:= [I_n,\,0]^T$ for a dynamic output-feedback controller. 
Clearly $M_0(A_{cl}) = \mathcal M_0(A_{cl})$ for static controllers. Note that
$\mathcal M_0(A_{cl})$ is generally not the same as $M_0(J^T A_{cl}J)$. 
The inequality $\mathcal M_0(A_{cl}) \leq M_0(A_{cl})$ follows from $\|J\|\leq 1$, so that
$\mathcal M_0(A_{cl}) \leq 1$ if  $e^{A_{cl} t}$ is  a contraction. Note, however, that we are not primarily interested in rendering $e^{A_{cl} t}$
contractive. Instead, we want to control the amplification of the $x$-part of the closed loop trajectories, so
that $\mathcal M_0(A_{cl})=1$ may occur even for non-contractive  $A_{cl}$.\newline

\noindent {\bf Example.}
A simple illustration of this possibility  is $A_{cl} = \begin{bmatrix}
-2 &0;3 & -1 \end{bmatrix}$  where for $J^T=[1\; 0]$, $\mathcal M_0(A_{cl})= \sup_{t\geq 0}\|J^T e^{A_{cl} t} J\|= \sup_{t\geq 0} |e^{-2t}| = 1$ 
whereas $M_0(A_{cl}) > 1$ because $A_{cl}$ has positive numerical abscissa $\omega(A_{cl})>0$, i.e., does not generate a contraction; see Lemma \ref{lemma2}.

Similarly, to assess the transient behavior of the closed loop, we may either use the Kreiss constant 
$K(A_{cl})$ directly, or again its restriction to the plant states only, by introducing
$$
\mathcal K(A_{cl}):= \max_{{\rm Re} (s) > 0}\; {\rm Re} (s) \, \|J^T(s I_{n+n_K}-A_{cl})^{-1}J \| \,,
$$
which in view of Theorem \ref{theo1} and the definition of $J$ above is expressed as
\begin{equation}\label{eqRobCL}
\mathcal K(A_{cl}) = \max_{\delta \in [-1,\,1]} \left\| J^T \left(s I - \left( \textstyle\frac{1-\delta}{1+\delta}  A_{cl} -I\right)\right)^{-1} J\right\|_\infty \,.
\end{equation}
For any fixed controller  this can be computed with the tools in Theorems \ref{theo1} and \ref{theo3}. 
For static controllers, $\mathcal K(A_{cl})=K(A_{cl})$, and clearly $\mathcal K(A_{cl}) \leq K(A_{cl})$ in general because of
$\|J\|\leq 1$.

Note that the analogue of the Kreiss matrix theorem for $\mathcal K(A_{cl})$ is obtained with little effort:
\begin{lemma}
\[
\mathcal K(A_{cl}) \leq \mathcal M_0(A_{cl}) \leq en \,\mathcal K(A_{cl}).
\]
\end{lemma}
\begin{proof}
For the left hand inequality, we take
\begin{align*}
    \|J^T (sI-A_{cl})^{-1} J \|&= \left\|\int_0^\infty e^{-st} J^T e^{A_{cl}t} J dt \right\|\\
    &\leq  \sup_{t\geq 0} \|J^T  e^{A_{cl}t} J\| \int_0^\infty e^{-{\rm Re}(s)t} dt
    = \mathcal M_0(A_{cl}) {\rm Re}(s)^{-1}.
\end{align*}
For the upper-bound,  we follow the argument in \cite{leveque1984resolvent} improved by \cite{spijker1991conjecture}. We have
for two test vectors $u,v$
\begin{align*}
    u^TJ^Te^{A_{cl}t}Jv &= \frac{1}{2\pi i} \int_{{\rm Re}(s)=\mu} \!\!\!\! e^{st} \underbrace{u^TJ^T(sI-A_{cl})^{-1} Jv}_{=: q(s)} ds \\
    &= 
    - \frac{1}{2\pi i} \int_{{\rm Re}(s)=\mu}
    \frac{e^{st}}{t} q'(s) ds 
    =
    -\frac{1}{2\pi i} \frac{e^{\mu t}}{t} \int_{\omega=-\infty}^{+\infty} e^{i \omega t} q'(\mu+i\omega) i d\omega.
\end{align*}
Hence if we let ${\rm Re}(s)=\mu = 1/t$ and take norms
\begin{align*}
    \left\|u^TJ^T e^{A_{cl}t}J  v \right\| &\leq  \frac{e}{2\pi} \frac{1}{t} \int_{-\infty}^\infty  |q'(1/t+i\omega)| d\omega 
    = \frac{e}{2\pi} {\rm Re}(s) \|q'({\rm Re}(s) + i \cdot)\|_1.
\end{align*}
Now \cite{spijker1991conjecture} improves the estimate of \cite{leveque1984resolvent}
to the extent that $\|q'\|_1 \leq 2\pi n \|q\|_\infty$, hence we get
\begin{align*}
|u^T  J^T e^{A_{cl}t}Jv| 
&\leq en {\rm Re}(s) \sup_{\omega} |u^TJ^T (({\rm Re}(s)+i \omega)I-A_{cl})^{-1} Jv| \\
&\leq en \sup_{{\rm Re}(s) > 0} {\rm Re}(s) | u^TJ^T (sI -A_{cl})^{-1} Jv|
\end{align*}
and since  $u,v$ are arbitrary, we get the right-hand estimate $\mathcal M_0(A_{cl}) \leq e\,n \, \mathcal K(A_{cl})$.
\hfill $\square$
\end{proof}

For the purpose of feedback synthesis, we have decided against the use of design techniques based on the LMI   characterization  in (\ref{eqpg2}). 
The reason is that the size of the scaling matrices $X$ and $Y$ grows as $O((n+n_K)^2)$ for an output feedback controller 
of order $n_K$
and most SDP solvers will succumb beyond $50$ states. The LMI approach (\ref{eqpg2}) shall be used only for certification. 
More precisely, once a controller has been synthesized, a lower bound of $\mathcal K(A_{cl})$ is obtained by the local optimizer. 
The exact value of $\mathcal K(A_{cl})$ 
at the final controller is then 
re-computed
via the methods of section \ref{sect:exact}, and thereby certified. 
Our experiments show that certification is practically always redundant, which corroborates what was already 
observed for the rich test sets in \cite{ANR:2015,AN:2015},
where uncertainty in several parameters and complex blocks was considered.

For synthesis,
we privilege the worst-case approach in (\ref{eqRob}) applied in closed loop. This leads to the min-max synthesis program
%
%
\begin{align}
\label{eqsynth}
\begin{array}{ll}
\displaystyle\mbox{minimize} & \displaystyle\max_{\delta \in [-1,1]} \left\| J^T \left( sI- \left(\textstyle \frac{1-\delta}{1+\delta} A_{cl}(K)-I \right)\right)^{-1} J \right\|_\infty\\
\mbox{subject to}& K \mbox{ stabilizing},  K \in \mathscr K \,,
\end{array}
\end{align}
where $K\in \mathscr K$ denotes a prescribed controller structure. 
This could for instance be PIDs, observed-based or low-order controllers, decentralized controllers,
as well as control architectures  assembling simple control components. 
Note that the stabilizing constraint on $K$ in (\ref{eqsynth})  enforces stability of the whole set of matrices 
$\left\{\frac{1-\delta}{1+\delta} A_{cl}-I: \;\delta \in [-1,1]\right\}$,  and in particular, that of
$A_{cl}(K)$ as desired.

In some cases it may be advisable to add further specifications on the
closed loop in (\ref{eqsynth}). Those may concern the parametric robust loop, the nominal loop, or elements of the loop,
like $K$, which would allow to distinguish further among multiple solutions of (\ref{eqsynth}).\\

\section{Algorithm \label{sect:optim}}
Using standard state augmentations
$${\small \begin{array}{l}
A_a = \begin{bmatrix} A & 0\\0 & 0_{n_K}\end{bmatrix}, \;
B_a = \begin{bmatrix} 0 & B \\ I_{n_K}&0\end{bmatrix},  \;
C_a = \begin{bmatrix} 0 & I_{n_K} \\ C & 0\end{bmatrix}, 
K_a = \begin{bmatrix} A_K & B_K \\ C_K & D_K\end{bmatrix}\,, x_a = \begin{bmatrix} x \\ x_K \end{bmatrix}\,,
\end{array}
}$$
and exploiting the open-loop state-space representation of $P$ in (\ref{eqP}),  the closed-loop system in program (\ref{eqsynth}) can be rewritten  in LFT form 
as $\delta I_{n+n_K} \star P_a \star K_a$ where $P_a$ has the state-space representation 
\begin{equation*}
P_a(s): \quad \left\{
\begin{aligned}
\dot x_a & =  (A_a - I_{n+n_K}) x_a + \sqrt{2} w_{\delta} + J w + B_a u \\
z_\delta & =  -\sqrt{2} A_a x_a - w_\delta - \sqrt{2} B_a u \\
z & =  J^T x_a \\
y & =  C_a x_a \,. 
\end{aligned}\right.
\end{equation*}
This means program (\ref{eqsynth}) may be recast as
\begin{equation}
\label{program}
\min_{K\in \mathscr K} \max_{\delta \in [-1,1]} \| \delta I \star P_a\star K_a\|_\infty.
\end{equation}

Note that program (\ref{eqsynth}), (\ref{program}) has three sources of non-differentiability. For fixed $\delta$ the $H_\infty$-norm
$\|\delta\star P_a\star K_a\|_\infty$ already
is non-smooth (a) due to the maximum singular value $\overline{\sigma}$, and (b) due to the semi-infinite maximum over
the frequency range $\omega\in [0,\infty]$. To this we have to add (c), the semi-infinite maximum over $\delta\in [-1,1]$, which
is the severest difficulty, because here a non-concave maximum has to be computed globally. 
To overcome this difficulty, we use the method of \cite{noll2019cutting,AN:2015,ANR:2015},
which we now briefly recall.

The basic idea is to select a small but representative set of scenarios
$\delta_\nu \in [-1,1]$, $\nu=1,\dots,N$, such that the multi-model $H_\infty$-synthesis program
\[
\min_{K\in \mathscr K}  \max_{\nu=1,\dots,N} \|\delta_\nu I \star P_a \star K_a\|_\infty
\]
gives an accurate estimation of the optimal value of (\ref{eqsynth}), resp. (\ref{program}).
This hinges on an intelligent selection of these worst-case scenarios, which we achieve by the scheme shown in Fig. \ref{fig-scheme}.
\begin{algorithm}[h!]
\title{Parametric robust synthesis}
\begin{algorithmic}[1]
\STEP{Initialize} Put $I=\{0\}$ and go to multi-model design.
\STEP{Multi-model} Given finite set $I \subset [-1,1]$ of scenarios, perform multi-model
$H_\infty$ (or $H_2$) synthesis
\[
h_*= \min_{K \in \mathscr K} \max_{\delta \in I}\|\delta I \star P_a \star K_a\|_{\infty,2}
\]
and obtain multi-scenario controller $K^*\in \mathscr K$.
\STEP{Destabilize} Compute worst-case scenario $\delta^* \in [-1,1]$ by solving
$$
\alpha^*=\max_{\delta\in [-1,1]} \alpha\left(\delta I \star P_a \star K_a^*\right).
$$
If $\delta^* I \star P_a \star K_a^*$ is unstable $(\alpha^* \geq 0)$, add $\delta^*$ to bad scenarios $I$ and go back to step 2.
Otherwise $(\alpha^* < 0)$ continue.
\STEP{Degrade} Compute worst-performance scenario $\delta^* \in [-1,1]$ by solving
$$
h^* = \max_{\delta\in [-1,1]} \|\delta I \star P_a \star K_a^*\|_{\infty, 2}.
$$
\STEP{Stopping}
If $h^* < (1+{\rm tol})h_*$ degradation is only marginal, then  accept $K_a^*$ and goto post-processing. Otherwise
add $\delta^*$ to bad scenarios $I$ and go back to step 2.
\STEP{Certify} Use method of section \ref{sect:exact} to certify final value $h^*$.
\end{algorithmic}
\end{algorithm}
The multi-scenario synthesis is performed efficiently using the method of \cite{AN:2015}, implemented in the MATLAB facility {\tt systune}. 
The fact that only a one-dimensional uncertainty cube is at work here,  as compared to the case of  \cite{noll2019cutting,AN:2015,ANR:2015},
is of course favorable in the present situation,
leading to fast and reliable estimates. The crucial observation is that programs $\alpha^*, h^*$ are of max-max type, whereas $h_*$ is
of min-max type. For more detail on how these characteristic differences are exploited algorithmically, see \cite{AN:2015}.

\begin{figure}[htbp]
\centering
\includegraphics[width = 0.8\columnwidth]{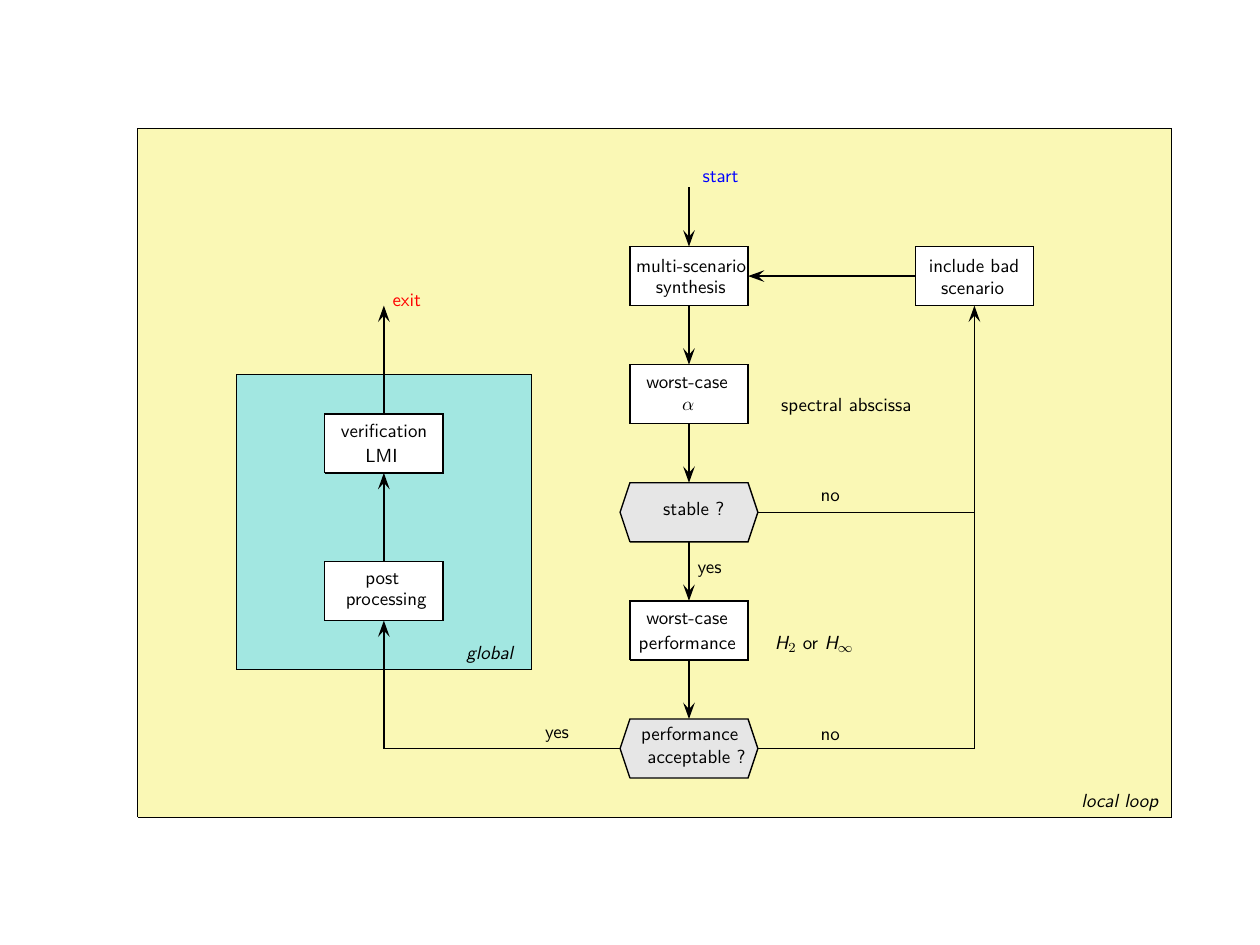}
\caption{Iterative selection  of bad scenarios such that multi-model synthesis for these covers the full uncertain range.} 
\label{fig-scheme}
\end{figure}

\section{Applications \& competing methods \label{sect:numerics}}
In this section, we consider minimization of the Kreiss constant in closed loop. The results are then compared
to a variety of other techniques, also allowing to reduce the effect of transients, possibly by less direct means. 
We work with an
example borrowed from \cite{whidborne2007minimization}. State-space data of the plant $G(s)=C (sI-A)^{-1}B$ in (\ref{eq-ABCD}) are given as

\begin{equation}\label{eqA}
{\small 
A = \begin{bmatrix}
-1      &     0      &     0       &    0    &       0     &      0    &    -625 \\
           0   &       -1    &     -30    &     400    &       0    &       0   &      250\\
          -2     &      0   &       -1    &       0    &       0     &      0      &    30\\
           5     &     -1        &   5     &    -1    &       0      &     0   &     200\\
          11     &      1    &      25    &     -10    &      -1    &       1  &      -200\\
         200     &      0      &     0     &   -150    &    -10^2   &       -1    &   -10^3\\
           1     &      0      &     0      &     0     &      0    &       0    &      -1
\end{bmatrix} }
\end{equation}
\begin{equation}\label{eqBCD}
{\small B = \begin{bmatrix}
I_4\\ 0_{3\times 4} 
\end{bmatrix}, \\
C = \begin{bmatrix} 0 &0 &0 &0 &0 &1 &0\end{bmatrix},\; D = 0_{1\times4}\,}.
\end{equation}
The plant has therefore four control inputs and a single measurement. 

In \cite{whidborne2007minimization}, the problem of transient growth minimization is approached using LMI techniques.
Large signal amplifications are constrained by reducing the eccentricity of the Lyapunov level-curves, 
where  Lyapunov function candidates are chosen as  quadratic functions $V(x) = x^T P x$. This is implemented as reducing the condition number of $P$, 
that is, minimizing $\gamma$ subject to $I \preceq P \preceq \gamma I$ in combination with additional closed-loop stability
constraints.  The Lyapunov derivative condition $\frac{d}{dt}V(x)\leq 0$ over all state  trajectories then ensures $x(t)\in\{\zeta \in \mathbb R^n |\zeta^T P \zeta \leq 1\}$ at all times $t\geq 0$, and for all initial conditions in that same set. 
The synthesis problem can be converted to a convex SDP at the price of using the Youla parameterization of stabilizing controllers \cite{BoB:91}.
This leads to controllers of the form
$K(s) = (I+Q(s)) G(s) Q(s)$ with $Q(s)$ the Youla parameter optimized over a finite Ritz basis subspace in $RH_\infty$. Controllers computed using this technique have order $n+ 2n_Q$, where $n_Q$ is the state dimension of $Q(s)$. In \cite{whidborne2007minimization}, a controller of order $7 + 2\times 9= 25$ was obtained with corresponding transient growth of 
$\sup_{t\geq 0} \| J^T e^{A_{cl}t} J\|=\sqrt{11919}= 109.2$. 

To allow for unbiased comparisons, all techniques discussed in the sequel are implemented
in their native formulation. In more practical applications, design programs should be complemented with more conventional control requirements such as  robust stability  margins or noise attenuation. The only exception to this rule is a constraint on the closed-loop spectrum as shown in Fig. \ref{fig-Disk}, to avoid excessively slow responses or much too high gain controllers. The latter constraint is of paramount importance, 
since pure performance design problems as in (\ref{eqsynth}) tend to generate unacceptable high-gain controllers. 

We have used restarts to improve local solutions. The very same $10$ starting points have been used for all techniques described in the sequel. The best over the $10$ local solutions is then retained for simulation and assessment. 
The controller structure $\mathscr K$ is specified as the set of $3rd$-order  controllers
for all approaches, which leads to 28 unknowns. 

All results are assessed via comparison with the open-loop transient growth $\sup_{t\geq 0}\|e^{A t}\|$ shown in Fig. \ref{fig-OL} (left).

\begin{figure}[htp]
\centering
\includegraphics[width = 0.5\columnwidth]{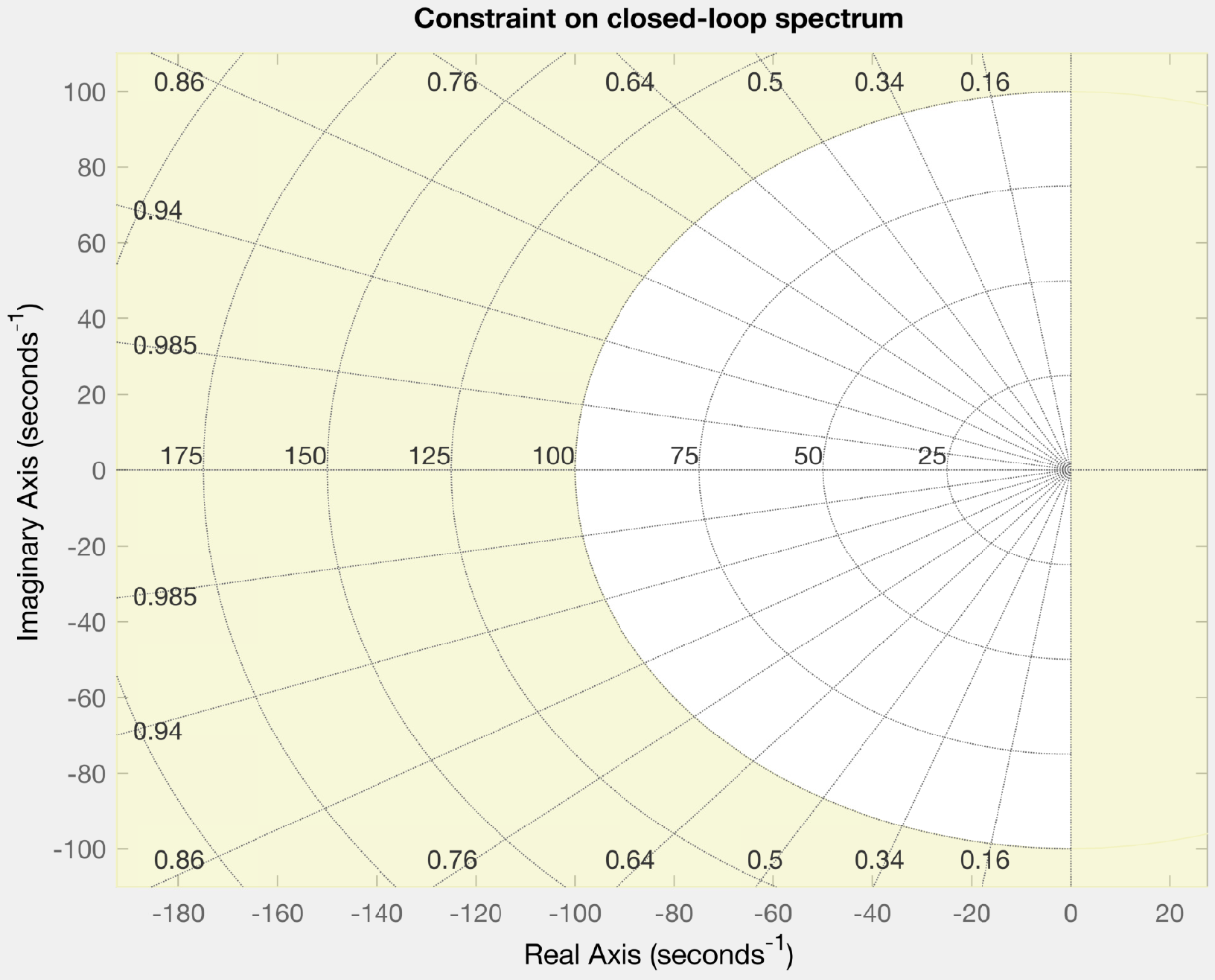}
\caption{Disk $D$ of closed-loop spectrum constraint  minimum decay of $0.001$ and disk constraint of radius $100$.} 
\label{fig-Disk}
\end{figure}

\subsection{Kreiss constant approach\label{subsec-Kreiss}}
For minimization of the Kreiss constant in feedback loop, the cast in (\ref{eqsynth}) is changed as  

\begin{align}
\label{eqsynth2}
\begin{array}{ll}
\displaystyle\mbox{minimize} & \displaystyle\max_{\delta \in [-1,1]} \left\| J^T \left( sI- \left(\textstyle \frac{1-\delta}{1+\delta} A_{cl}(K)-I \right)\right)^{-1} J \right\|_\infty\\
\mbox{subject to}& K \mbox{ stabilizing},  K \in \mathscr K \\
& \sigma(A_{cl}(K)) \in D \,,
\end{array}
\end{align}
with $\sigma(A_{cl}(K))$ denoting the spectrum of $A_{cl}(K)$.

The best controller over $10$ restarts is obtained as 
$$K(s)= \footnotesize
\left[\begin{array}{lll|l}                                
-42.9038 & 11.5813 & 0.0000 & 0.0128 \\      
-164.9255 & 70.3235 & 152.7735 & -13.6539 \\ 
0.0000 & -25.9407 & -149.4428 & 11.8197 \\  \hline 
-167.0674 & 318.3261 & 809.8411 & -66.1531 \\
200.4722 & 413.5407 & -666.0200 & 72.5131 \\ 
-66.2768 & 27.6020 & 76.9643 & -2.4021 \\    
385.9815 & -189.8190 & -229.6792 & 22.6246 
\end{array}  \right]                              
$$
with the standard notation 
$$ K(s) = C_K(sI_{n_K} - A_K)^{-1} B_K + D_K = 
\left[\begin{array}{l|l} 
A_K & B_K \\ \hline C_K & D_K\end{array}  \right] \,,
$$
and its transient growth is shown in Fig. \ref{fig-OL} (right), 
with a peak value of $42.8$. This improves over the higher-order LMI controller of \cite{whidborne2007minimization}, which achieves 109.2. Our solution gives a 
reduction by one order of magnitude  over the open-loop transient growth $680.4$ displayed in Fig. \ref{fig-OL} (left). 
The closed-loop Kreiss constant  computed via program (\ref{eqsynth2})
is $10.90$, which we certified as $10.91$ using the exact approach in Theorem \ref{theo3}. 
Program (\ref{eqsynth2}) was solved using {\it systune} based on \cite{AN:2015,GA2011a,ANtac:05,BoApNo:07} from The Control System Toolbox of MATLAB, 
while the certificate was computed using the routine {\it wcgain} from The Robust Control Toolbox. 

\begin{figure}[htp]
\centering
\includegraphics[width = 6cm,height=5cm]{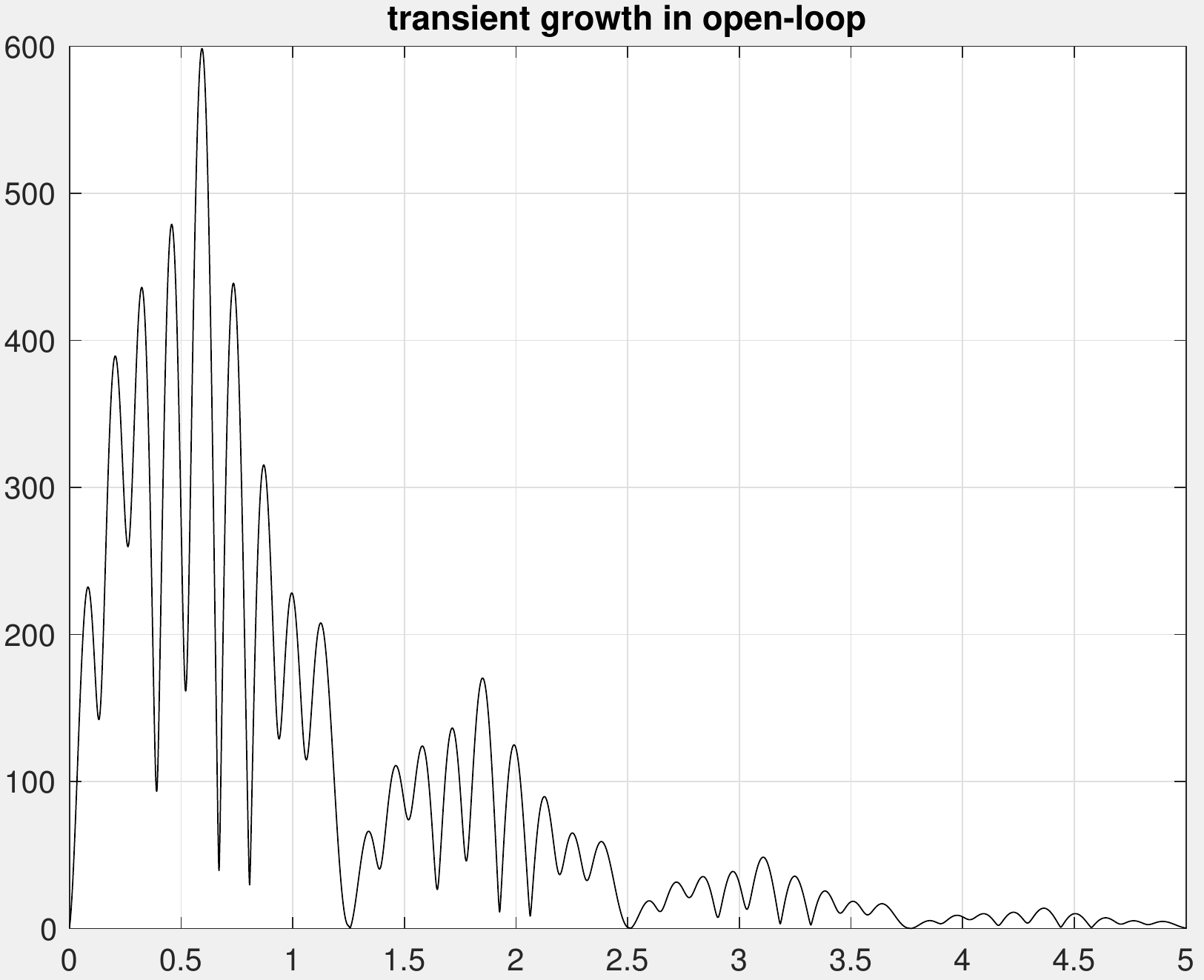}
\includegraphics[width = 6cm,height=5cm]{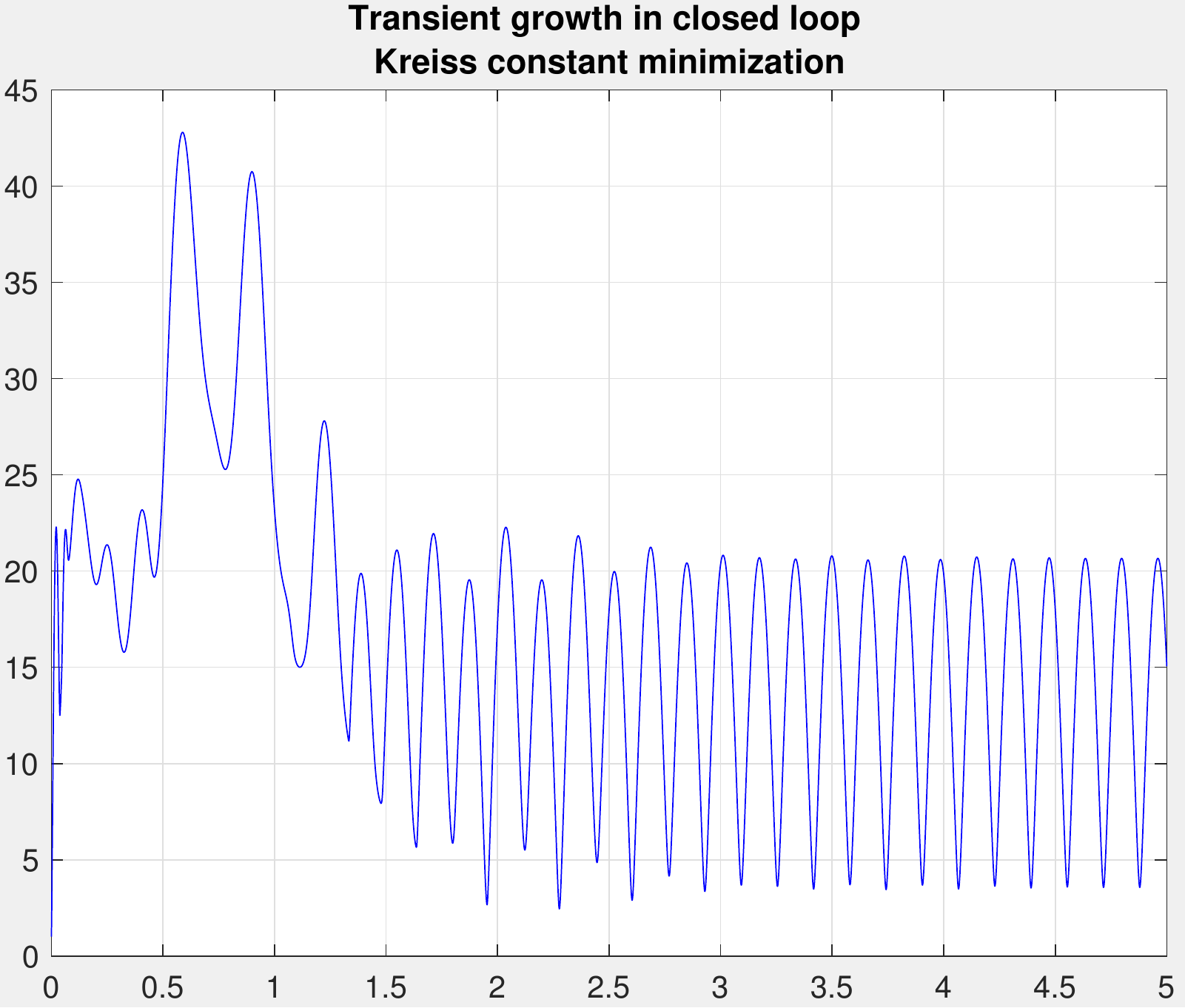}
\caption{Transient growth in open loop (left) and in closed loop (right) by minimizing the Kreiss constant.} 
\label{fig-OL}
\end{figure}


\subsection{Numerical abscissa approach} \label{subsec-NA}
The numerical abscissa of (\ref{eqPlant1}) is defined as 
$$\omega(A):= \textstyle\frac{1}{2}\,\overline{\lambda}(A+A^T)$$ where $\overline{\lambda}$ stands for the maximum eigenvalue of a symmetric matrix. 
The central  properties of  the numerical abscissa are summarized by the following
\begin{lemma}\label{lemma2} Consider a possibly unstable autonomous system {\rm (\ref{eqPlant1})}. Then the following hold: 

\begin{itemize}
    \item[(a)] The  transient growth satisfies $\|e^{At}\| \leq 1$ for all $t\geq 0$  iff 
    $\omega(A)\leq 0$. 
    \item[(b)] For every $t\geq 0$, 
    $\|e^{At}\| \leq e^{\omega(A) t}$ \,.
    \item[(c)] In the limit we have $$\lim_{t\downarrow 0}\frac{d}{dt}\|e^{At}\| = \omega(A)\,.$$
    \item[(d)] If $A$ is normal, then $\omega(A)=\alpha(A)$. 
\end{itemize}
\end{lemma}
\begin{proof}
Proofs in various forms can be found in \cite{trefethen2005spectra,whidborne2007minimization,hinrichsen2000transient}. 
\hfill $\square$
\end{proof}

Property (a) gives a simple computational  test whether $A$ generates a contraction, hence whether $K(A)=1$. 
Property (c) indicates that the numerical abscissa determines the behavior of the transient growth as $t\to 0$, that is, in a short
time range. Property (b) on the other hand suggests that transient growth at intermediate times might also to some extent
be contained by making  the numerical abscissa as small as possible. \newline

\noindent
{\bf Example.} Strong dissipativity $A + A^T < 0$ implies $K(A)=1$
by condition (a) in Lemma \ref{lemma2}.
For upper triangular $2\times 2$ matrices $A = [a \; b ; 0 \; c]$ a necessary and sufficient condition for strong dissipativity
is $a < 0$ and  $4ac-b^2 > 0$.
This easily leads to non-normal matrices with $K(A) = 1$.

The numerical abscissa has been used in numerous studies and specifically in fluid  flow analysis to assess transition to turbulence, instabilities  and  limit cycles \cite{CGR:2009}. 
This suggests considering  the following indirect approach  to mitigate transient growth
of the plant state $x$ in closed loop: 
\begin{align}
\label{eqsynthNA}
\begin{array}{ll}
\mbox{minimize} &  \Omega(A_{cl}):=\omega\left(J^T A_{cl}(K) J\right)  \\
\mbox{subject to}& K \mbox{ stabilizing},  K \in \mathscr K \\
& \sigma(A_{cl}(K)) \in D \,.
\end{array}
\end{align}

This is an eigenvalue optimization program, which can in principle be solved
using BMI techniques \cite{LM:99,KS:2003},  but again we privilege a nonsmooth approach as in \cite{ANtac:05} thereby avoiding size inflation due to Lyapunov variables.

A closed-loop numerical abscissa of $\Omega(A_{cl})=502.0$ was achieved, thus improving over
the open-loop value of $\omega(A)=680.4$. Naturally, the optimal controller of (\ref{eqsynthNA}) has a lower
closed-loop numerical abscissa than the
Kreiss controller in section \ref{subsec-Kreiss}, which gave the numerical abscissa of  $656$. However,
as can be observed in Fig. \ref{fig-CLNA} (left), minimization of the numerical abscissa did 
not achieve the desired effect of limiting the transient growth. The controller of (\ref{eqsynthNA})
did not even improve  over the open-loop behavior in Fig. \ref{fig-OL}. 
Those results are in line with the qualitative analysis \cite{trefethen2005spectra}, which identifies the numerical abscissa as a good indicator 
only for  $t\to 0$. 

The locally optimal $3$rd-order controller for program (\ref{eqsynthNA}) is given  as
$$K(s)= \footnotesize
\left[\begin{array}{lll|l} 
59.9714 & 140.8838 & 0.0000 & 125.1870 \\    
                                  
100.3809 & 151.4666 & -0.9285 & 152.6506 \\  
                                      
0.0000 & -271.6638 & -612.4505 & 514.0162 \\ \hline 
                                      
-180.2674 & 2.4115 & 610.7701 & -818.7354 \\ 
                                  
-1.9939 & 17.2208 & 896.7905 & 248.2384 \\   
                                     
134.2585 & 322.4479 & 198.7380 & 27.7581 \\  
                                      
145.1514 & 114.7305 & -229.1801 & 350.4296
\end{array} \right]
$$

\begin{figure}[htp]
\centering
\includegraphics[width = 5cm, height=5cm]{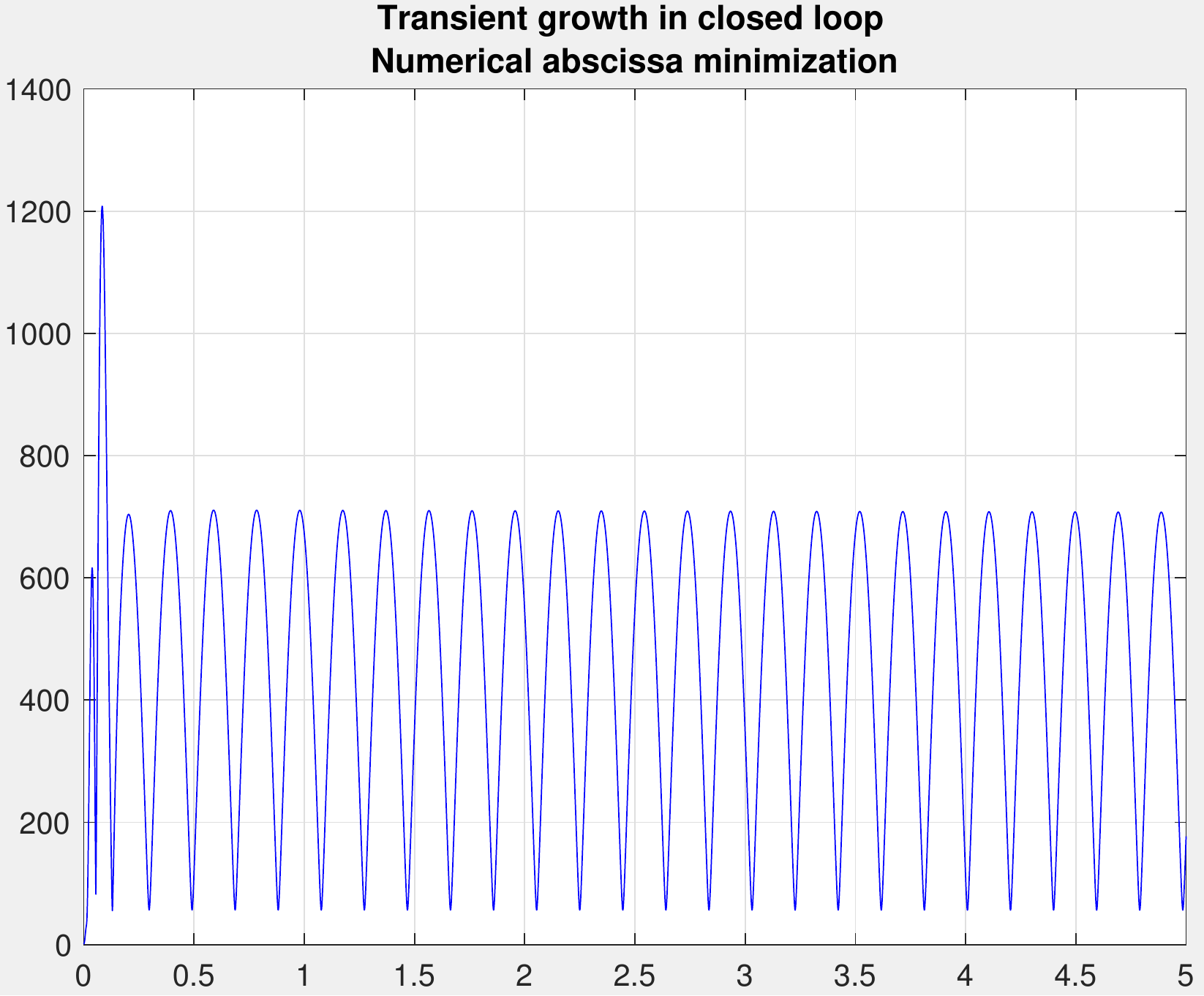}
\includegraphics[width = 5cm, height=5cm]{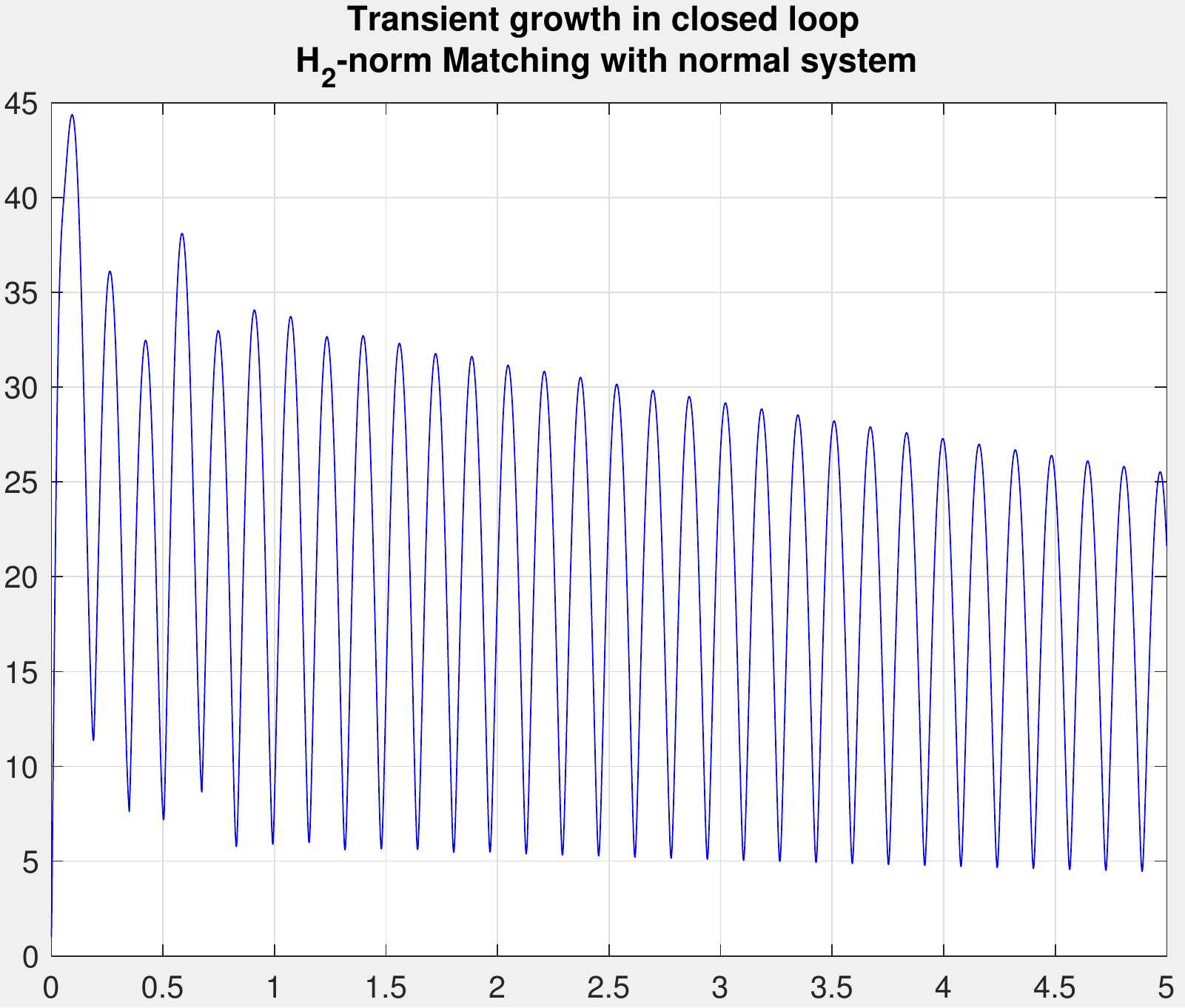}\includegraphics[width = 5cm,height=5cm]{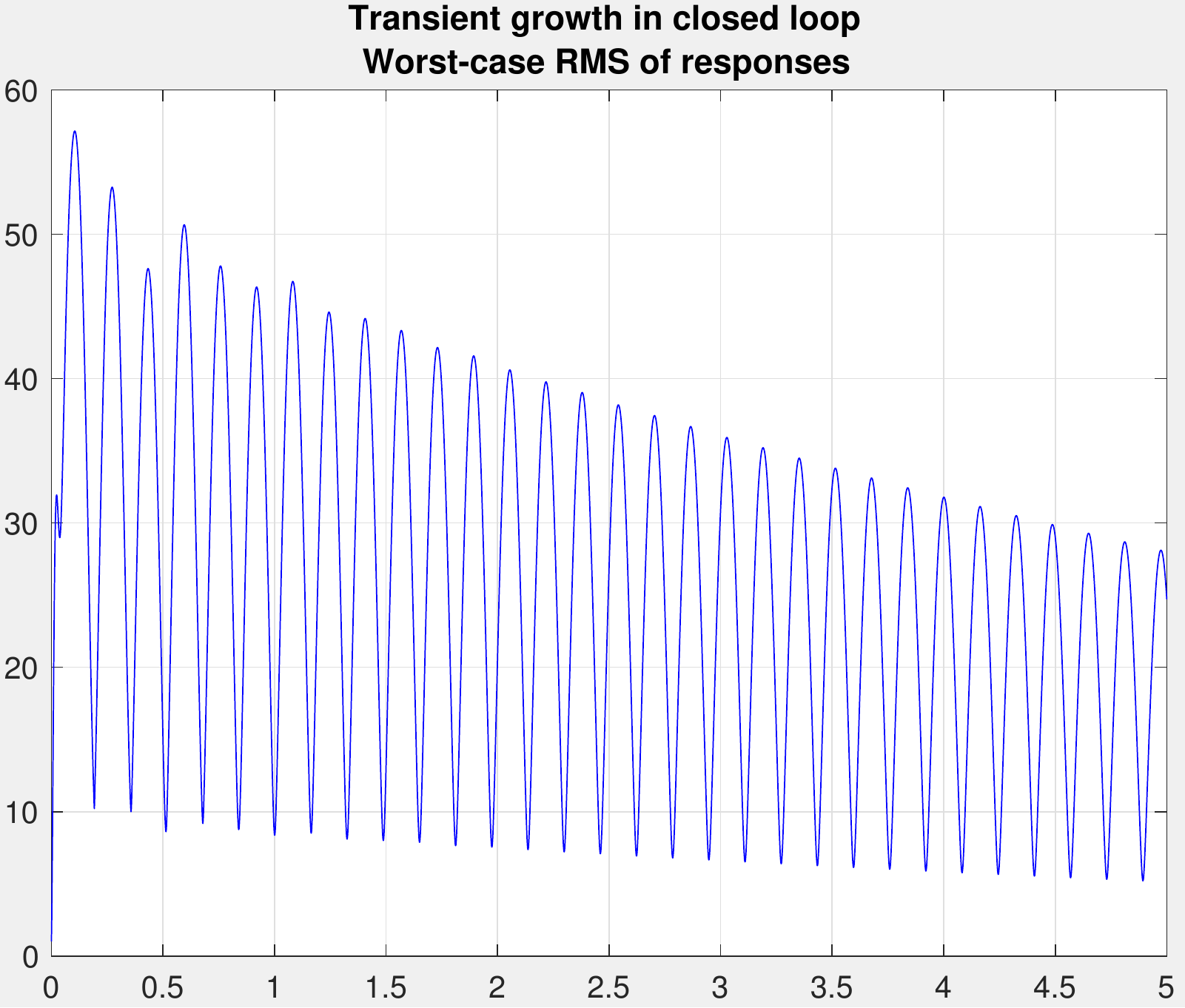}
\caption{Transient growth in closed loop.  Minimization of numerical abscissa (\ref{eqsynthNA}) left.
 $H_2$-norm matching with normal  model (\ref{eqsyntH2}) middle. Worst-case energy response (\ref{eqWCH2}) right.} 
\label{fig-CLNA}
\end{figure}

\subsection{$H_2$ model matching with normal dynamics\label{subsec-H2}}
In this section we discuss yet another method to constrain transient growth in closed loop.
For given initial conditions $x(0)=x_0$, the closed-loop state responses of the plant $G_{cl}$  are described by
\begin{align*}
     \dot x_{cl} & =  A_{cl}(K) x_{cl} + J w,\quad w = x_0 \delta(t)  \\
     z & =  J^T x_{cl}\,.
\end{align*}
%
By tuning the controller $K\in \mathscr K$,
we would like this system to behave similar to an ideal reference system $G_r(s)$ deliberately constructed to exhibit small transient growth, say,  
%
\begin{align*}
     \dot x_r & = A_r x_r +  w_r,\quad w_r = x_r^0 \delta(t)  \\
     z_r & = x_r \,.
\end{align*}
This leads to a model matching optimization problem,
where we minimize the mismatch $z-z_r$ between the responses of both systems, 
started from the same initial conditions $w = w_r = x_0 \delta(t)$. 
If $z-z_r$ is measured in the energy norm,
this leads to
\begin{align*}
\|z-z_r\|_2 & =  \|G_{cl}x_0 \delta(t)-G_r x_0 \delta(t)\|_2 
\leq  \|G_{cl}-G_r\|_2 \|x_0\| \,,
\end{align*}
where for systems $\|G-G_r\|_2$ means the $H_2$-norm.
Consequently, we consider the following cast:
\begin{align}
\label{eqsyntH2}
\begin{array}{ll}
\displaystyle\mbox{minimize} & \|J^T(sI-A_{cl}(K))^{-1}J - (sI-A_r)^{-1}\|_2\\
\mbox{subject to}& K \mbox{ stabilizing},  K \in \mathscr K \\
& \sigma(A_{cl}(K)) \in D \,,
\end{array}
\end{align}
where as before, one enforces structural constraints on the controller $K \in \mathscr K$, 
and spectral constraints $\sigma(A_{cl}(K)) \in D$ on the loop, ruling out slow responses and much too high gain controllers, respectively. 

This indirect approach to transient growth mitigation is illustrated for the system in (\ref{eqA})-(\ref{eqBCD}), 
where the reference model is selected with  normal dynamics $ G_r(s) = (s I  - (-I))^{-1}$ and numerical abscissa 
$\omega(A_r) = \alpha(A_r)=  -1$.  With $\mathscr K$ the set of $3$rd-order controllers, and the semi-disk $D$ unchanged as in Fig. \ref{fig-Disk}, 
solving program (\ref{eqsyntH2}) leads to the controller 
$$K(s)= \footnotesize
\left[\begin{array}{lll|l}                     
-10.0166 & 32.8652 & 0.0000 & 4.2887 \\        
-5.3332 & -75.2766 & 74.7646 & 83.9716 \\      
0.0000 & 246.4755 & -258.5282 & -246.5133 \\   \hline                                         
-205.9510 & 236.5090 & -123.3962 & -152.2283 \\
-1153.0456 & -879.8479 & -71.1224 & 150.9151 \\
-13.2672 & -120.1666 & 21.8126 & 115.6246 \\   
21.5530 & 3.7044 & 60.9500 & 127.7649        
\end{array} \right].
$$

%

\noindent
The associated transient growth $\|J^Te^{A_{cl}t}J\|$ in closed-loop is shown in Fig.  \ref{fig-CLNA} (middle),  
with peak value
$\mathcal M_0(A_{cl})=44.37$, which indicates that this indirect approach is competitive with the Kreiss constant minimization. 
Even better results might be obtained by using a more plausible reference model $G_r$, 
but  this has not been pursued further in this work.

\subsection{Worst-case energy response approach} \label{subsec-WCRMS}

In this section, we change metrics and replace 
$
\max_{\|x_0\| \leq 1} \sup_{t\geq 0}\|x(t)\| = \max_{\|x_0\|_2 \leq 1} \|x\|_\infty
$
with the new norm
$
\max_{\|x_0\|_\infty \leq 1} \|x\|_2 = \max_{|x_{0,i}| \leq 1,\, i= 1,\ldots,n} \sqrt{\int_{0}^{\infty} x(t)^T x(t) dt} \,,
$
and investigate whether the substitute has some merit in
reducing transient growth in closed loop with output feedback.

The closed-loop formulation in state-space is now given by the system: 
\begin{align*}
\dot x_{cl} & =  A_{cl} x_{cl} + J w   \\
 z    & =  J^T x_{cl}\qquad (= x) \\
 w  & =  x_0 \,\delta(t), \;\, \|x_0\|_\infty \leq 1\,.
\end{align*}

\noindent
This in turn leads to the minimization problem 
\begin{align}
\label{eqWCH2}
\begin{array}{ll}
\mbox{minimize} & \displaystyle\max_{\|x_0\|_\infty \leq 1} \left\| J^T \left( sI-   A_{cl}(K) \right)^{-1} J x_0 \right\|_2\\
\mbox{such that}& K \mbox{ stabilizing},  K \in \mathscr K \\
& \sigma(A_{cl}(K)) \in D \,,
\end{array}
\end{align}
which is similar in nature to the worst-case performance problem of the Kreiss constant approach in (\ref{eqsynth2})
and can be solved  with the same  techniques. 

For fixed $K$, program (\ref{eqWCH2}) has a certificate in terms of a convex SDP. To see this, we note first that the state-space data in (\ref{eqWCH2}) 
range over a matrix polytope
$$
\left\{ \sum_{i=1}^{2^n} \theta_i \begin{bmatrix} A_{cl} & J v_i \\ J^T & 0
\end{bmatrix}:\; \sum_{i=1}^{2^n} \theta_i = 1, \; \theta_i \geq 0 \right\}\,,
$$
where the $v_i$'s, $i = 1,\ldots,2^n$ denote the vertices of the unit cube $[-1,1]^n$. The optimal value of program (\ref{eqWCH2}) is then  
$< \gamma$  if an only if  there exist a  Lyapunov matrix function $X(v)=X(v)^T\succ 0$ with $v = \sum_{i=1}^{2^n} \theta_i v_i$, $\theta_i \geq 0$,
$\sum_{i=1}^{2^n} \theta_i = 1$, such that  
\begin{equation}\label{eq-v}
\begin{array}{l}
\begin{bmatrix}
A_{cl} X(v) + X(v) A_{cl}^T &  J v \\ (\bullet)^T & -1
\end{bmatrix} \prec 0, \quad
{\rm Tr}(J^T X(v) J) < \gamma^2,\qquad \forall v \in [-1,1]^n \,.
\end{array}
\end{equation}
In particular, taking $v=v_i$ and denoting $X_i:= X(v_i)$, this implies 

\begin{equation}\label{eq-vi}
\begin{bmatrix}
A_{cl} X_i + X_i A_{cl}^T & J v_i \\ (\bullet)^T & -1
\end{bmatrix} \prec 0, \;
{\rm Tr}(J^T X_i J) < \gamma^2,\qquad  i = 1,\ldots,2^n\,.
\end{equation}

Conversely, taking convex combinations of the inequalities in (\ref{eq-vi}) shows that $X(v) = \sum_{i=1}^{2^n} \theta_i X_i$ is a suitable Lyapunov matrix for which (\ref{eq-v}) holds. 

We have thus established that certification of $H_2$ performance $\gamma$ reduces to
constraints at the vertices and can be done by solving the SDP:
\begin{align*}
    \begin{array}{ll}
    \mbox{minimize} & \gamma^2 \\
    \mbox{subject to}& \begin{bmatrix}
A_{cl} X_i + X_i A_{cl}^T & J v_i \\ (\bullet)^T & -1
\end{bmatrix} \prec 0, \\ 
&X_i=X_i^T \succ 0, {\rm Tr}(J^T X_i J) < \gamma^2,\qquad  i = 1,\ldots,2^n\,
\end{array}    
\end{align*}
with decision variables $X_i, \gamma$.
%
See \cite{elghaoui92_1} for 
examples of polytopic linear differential inclusions. 
Again one has to stress that such a certificate may be too expensive even for medium size applications due to the limitation of current SDP solvers. 

With the same starting points, controller structure $\mathscr K$ and semi-disk $D$, a solution $K(s)$ to program 
(\ref{eqWCH2}) was obtained as 
$$K(s)= 
\footnotesize 
\left[
\begin{array}{lll|l}                    
-11.5489 & 78.9907 & 0.0000 & 53.2452 \\      
199.9054 & -357.8574 & 329.8169 & -206.2099 \\
0.0000 & -60.0656 & -22.2754 & -40.3642 \\    \hline                                        
-136.5439 & -7.6336 & 193.7006 & 30.1711 \\   
-1434.8960 & 269.1622 & -473.4523 & 27.1643 \\
-482.9145 & 868.9921 & -824.9746 & 499.5364 \\
-39.9217 & 559.8141 & -80.1572 & 351.7574 \\  
\end{array}   \right].
$$

The transient growth in closed-loop is presented in Fig. \ref{fig-CLNA} (right), 
indicating that this alternative technique, while inferior to the Kreiss approach with a peak transient growth of 
$\mathcal M_0(A_{cl})= 57.1$, and closed-loop Kreiss constant of $\mathcal K(A_{cl})= 24.8$, may be a valid alternative. 


All results obtained so far are presented in table \ref{tabAll}. Synthesis based on the Kreiss constant  is clearly the best approach in terms of peak value amplification at the expense of longer computational times. 

\begin{table}[htbp]
\caption{Summary of results  in closed-loop:  transient growth $\mathcal M_0$,  Kreiss certificate $\mathcal K$, numerical abscissa $\Omega $ and mean running time per run in sec.}
    \label{tabAll}
\begin{center}
\begin{tabular}{|| l|l|l|l|l ||}
\hline\hline
 & $\mathcal M_0$  & $\mathcal K$ & $\Omega $ & cpu \\ \hline
section \ref{subsec-Kreiss}& $42.8$     & $10.91$  & $656$ & 32  \\ \hline
section \ref{subsec-NA}& $1208$    & $349.6$  & $502$ &  1.3  \\ \hline
section \ref{subsec-H2}& $44.37$    & $23.5$  & $621$  &  1.7 \\ \hline
section \ref{subsec-WCRMS}& $57.1$    & $24.8$  & $686$ & $4.7$  \\ \hline\hline
\end{tabular}
\end{center}
\end{table}

\subsection{A nonlinear example} \label{subsec-NL}
In this section, we illustrate how optimizing the Kreiss constant can be exploited
to mitigate adverse effects of nonlinearities.   The example is borrowed from \cite{trefethen1993hydrodynamic} and was used to illustrate how non-normality in the linear 
portion of the system can trigger nonlinear effects and may generate convergence to  undesired critical points. It has been complemented by one actuator and one sensor so that 
feedback control becomes applicable. 

The system dynamics are given as 
\begin{align}\label{eq-NLSyst}
\begin{split}
\dot x & = A x + \|x\| B_x x + Bu \\ 
y & = C x 
\end{split}
\end{align}
with 
{\small \begin{align*}
\begin{split}
A &= \begin{bmatrix}
-1/R & 1 \\ 0 &-2/R
\end{bmatrix},\;  B_x = \begin{bmatrix}
0 & -1 \\1 & 0
\end{bmatrix}, \; B = \begin{bmatrix}
1\\1
\end{bmatrix},\; 
 C = \begin{bmatrix}
1 & 0
\end{bmatrix},\; R = 25\,.
\end{split}
\end{align*}
}

The linear dynamics are indeed non-normal with Kreiss constant $K(A)= 4.36$, and according to section \ref{sect:exact}, one can anticipate significant transient growth. This is confirmed in Fig. \ref{figNL} left for a set of initial conditions $x_0 = \begin{bmatrix}
0 & x_2(0)
\end{bmatrix}^T$, with $x_2(0) \in \left\{1\mathrm{e}{-7}, 1\mathrm{e}{-6}, 1\mathrm{e}{-5}, 1\mathrm{e}{-4}, 4\mathrm{e}{-4}, 5\mathrm{e}{-4}, 1\mathrm{e}{-3}, 1\mathrm{e}{-2} \right\}$.

According to \cite{trefethen1993hydrodynamic}, the open-loop system converges to a remote  unexpected critical point for $x_2(0) > 4.22\mathrm{e}{-4}$ which evokes a butterfly effect with big consequences.  See Fig. \ref{figNL} right. 

In an attempt to mitigate undesirable nonlinear effects, we minimize the Kreiss constant 
as discussed in section \ref{sect:feedback}. The program is again (\ref{eqsynth2}) with disk constraints of Fig. \ref{fig-Disk} unchanged and using for $\mathscr K$ the set of $2$nd-order controllers. This gives the following controller and corresponding $A$-matrix:

$$
{\small
K(s)= 
\left[
\begin{array}{ll|l}    
-3.4146 & -0.1902 & -1.7997 \\
-0.2856 & -2.6781 & -0.1119 \\ 
\hline                        
-1.8068 & -0.1095 & -1.3710 \\
\end{array}   \right] , }\;
{\small A_{cl} = 
\begin{bmatrix}             
-1.4110 & 1.0000 & -1.8068 & -0.1095 \\ 
                                
-1.3710 & -0.0800 & -1.8068 & -0.1095 \\
                                
-1.7997 & 0.0000 & -3.4146 & -0.1902 \\ 
                                 
-0.1119 & 0.0000 & -0.2856 & -2.6781 \\ 
\end{bmatrix}        
}                   
$$
with nearly unit closed-loop Kreiss constant $\mathcal K(A_{cl}) \approx 1$ where $J = [I_2 \, 0]^T$. 
This is confirmed in Fig. \ref{figNL} (middle),  where identical plant-state initial conditions now converge monotonically  to the zero equilibrium as desired.

\begin{figure}[htp]
\centering
\includegraphics[width = 8cm, height=6cm]{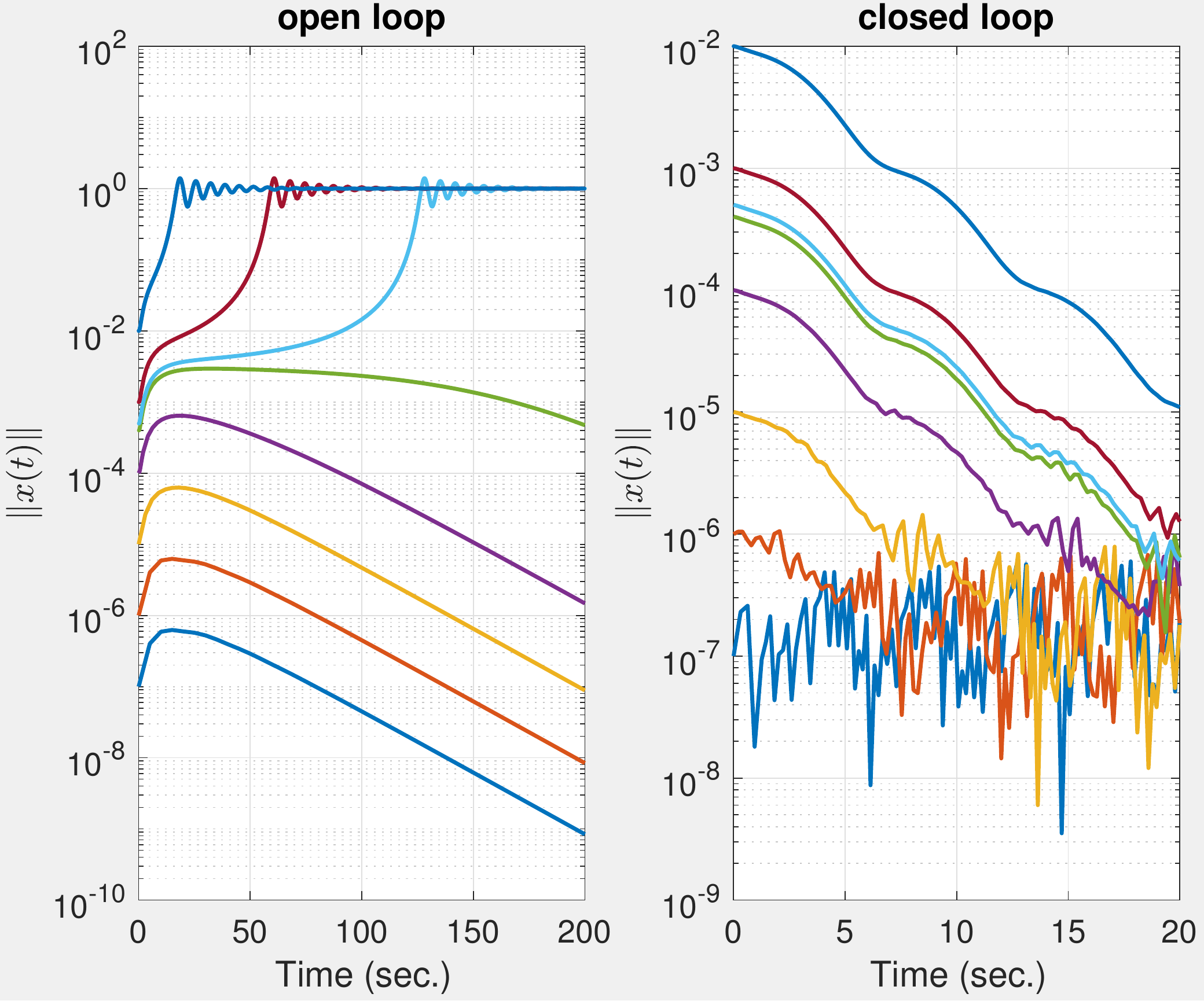}$\quad$
\includegraphics[width = 6cm, height = 6cm]{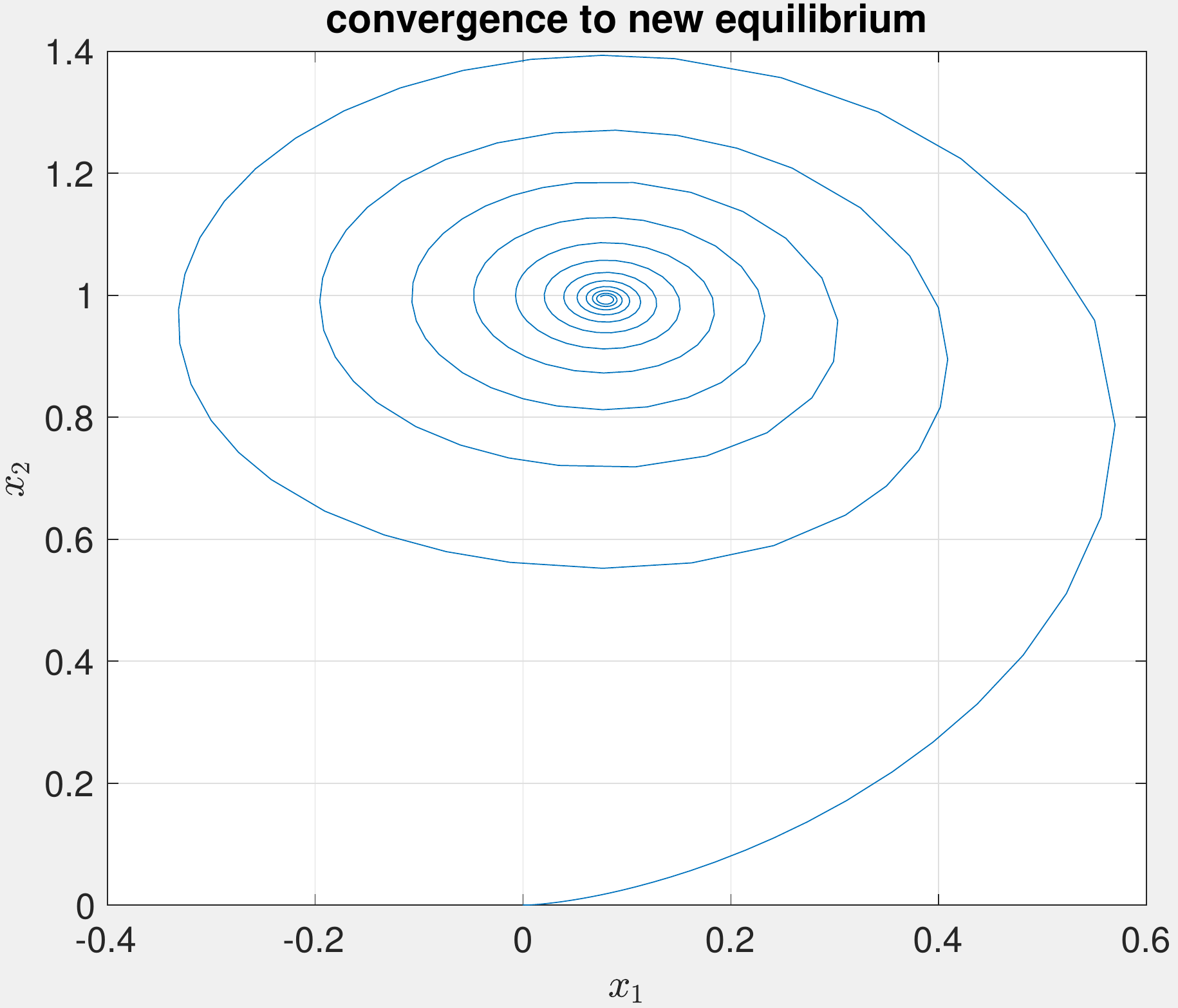}
\caption{Simulations of nonlinear system. Left: open loop. Middle: closed loop. 
Right: open-loop phase portrait for $x_0= [0;\, 5\mathrm{e}{-4}]$} 
\label{figNL}
\end{figure}

\section{Conclusions}
In this work, we have introduced a new exact computational technique for the Kreiss constant which essentially reduces to solving a 
robust performance analysis problem
of low complexity accessible to fairly standard $\mu$ tools. 
The new characterization is then further exploited by minimizing the Kreiss constant in closed loop,
with the goal to mitigate transient growth of potentially highly 
non-normal dynamics by the use of feedback. This leads to a special class of parametric uncertain structured $H_\infty$-control problems that are conveniently 
addressed with specialized non-smooth optimization 
methods. 
The development of mixed methods using jointly the Kreiss constant (peak growth), the numerical abscissa (initial growth), and the spectral abscissa (asymptotic growth)  to better shape the transient behavior is easily derived as a byproduct of this work.

A number of  comparisons have been made with indirect, more heuristic techniques. Our preliminary testing indicates that while seemingly conservative in the Kreiss matrix Theorem, the Kreiss constant can be an effective indicator of transient growth, and can be used to reduce it in closed loop. 
Some of the indirect approaches to transient growth,
even though suboptimal in theory, could constitute  valuable and  less costly alternatives. 

The LMI technique in section \ref{subsec-Kreiss} is suited for small to medium size problems.
For large scale problems more dedicated calculation methods will be required. 
This is in particular true for the $\mu$-certificate, which is a very general technique covering a wide class of problems, but leads to LMI-programs,
which are currently not fit for large dimensions. This is why in larger dimensions certification for $K(A)$ is performed with the method
of section \ref{sect:optim}, which is functional for systems up to several hundred states.
Specialization  to the Kreiss constant computation program (\ref{eqpg2}), which features a single repeated parameter  uncertainty 
is expected to render certification fit for even larger dimensions.

\bibliographystyle{plain}
\bibliography{DATABASE,biblio}
\end{document}